\documentclass[11pt]{article}

\usepackage{amsfonts,amsthm,amssymb,amsmath}
\usepackage{graphicx,multirow,dsfont}
\usepackage[usenames,dvipsnames]{color}

\usepackage{hyperref}

\usepackage[top=1in,bottom=1in,left=1in,right=1in 
]{geometry}
\numberwithin{equation}{section}
\numberwithin{table}{section}
\numberwithin{figure}{section}
\newtheorem{lemma}{Lemma}[section]
\newtheorem{remark}{Remark}[section]
\newtheorem{theorem}{Theorem}[section]
\newtheorem{assumption}{Assumption}[section]
\newtheorem{corollary}{Corollary}[section]

\newtheorem{proposition}{Proposition}[section]

\usepackage[round]{natbib}
\usepackage{authblk}

\usepackage{tikz}
\usetikzlibrary{shapes,arrows}

\newcommand{\R}{\mathbb{R}}

\newcommand{\Z}{\mathbb{Z}}

\newcommand{\fl}[1]{\lfloor{#1}\rfloor} 
\newcommand{\id}[1]{\mathds{1}_{\{{#1}\}}} 

\newcommand{\pr}[1]{\mathbb{P}\left\{#1\right\}}
\newcommand{\prob}{\mathbb{P}} 
\newcommand{\E}{\mathbb{E}}    

\newcommand{\dto}{\Rightarrow} 

\newcommand{\D}{\mathbf{D}} 
\newcommand{\C}{\mathbf{C}} 

\newcommand{\fs}[1]{\bar{#1}^n} 
\newcommand{\ds}[1]{\tilde{#1}^n} 

\usepackage[normalem]{ulem}

\newcommand{\hilit}[1]{\textcolor{red}{{#1}}}

\title{A Unified Approach to Diffusion Analysis of Queues with General Patience-time Distributions}

\author[*]{Junfei Huang}
\author[$\dagger$]{Hanqin Zhang}
\author[$\ddag$]{Jiheng Zhang}
\affil[*]{The Chinese University of Hong Kong}
\affil[$\dagger$]{National University of Singapore}
\affil[$\ddag$]{The Hong Kong University of Science and Technology}

\date{\today}

\begin{document}

\maketitle

\begin{abstract}
We propose a unified approach to establishing diffusion approximations for queues with impatient customers within a general framework of scaling customer patience time.
The approach consists of two steps.
The first step is to show that the diffusion-scaled abandonment process is asymptotically close to a function of the diffusion-scaled queue length process under appropriate conditions.
The second step is to construct a continuous mapping not only to characterize the system dynamics using the system primitives, but also to help verify the conditions needed in the first step.
The diffusion approximations can then be obtained by applying the continuous mapping theorem.
The approach has two advantages:
(i) it provides a unified procedure to establish the diffusion approximations regardless of the structure of the queueing model or the type of patience-time scaling;
(ii) and it makes the diffusion analysis of queues with customer abandonment essentially the same as the diffusion analysis of queues without customer abandonment.
We demonstrate the application of this approach via the single server system with Markov-modulated service speeds in the traditional heavy-traffic regime and the many-server system in the Halfin-Whitt regime and the non-degenerate slowdown regime.
\end{abstract}


\section{Introduction}
\label{sec:introduction}

Motivated by its frequent occurrence in many service systems, customer abandonment has been extensively studied in various queueing models.
For example, outstanding orders may be canceled in manufacturing industries, data packets may be dropped if the waiting time in the transmission channel is too long, and customers may hang up at a call center after waiting for a while.
Abandonment is modeled by assuming each customer (order, data packet, etc.) has a patience time, which is a random variable.
A customer abandons the system once his waiting time exceeds his patience time.
The study of customer abandonment dates back to \cite{Palm1937}, who noticed the impatient behavior of telephone switchboard customers.
Many studies focus on the diffusion analysis of queueing processes as they often yield tractable and meaningful approximations.
This paper aims to provide a unified approach to diffusion analysis with general patience-time distributions.

In the literature, there are two main streams of studies on abandonment that differ by the patience-time scaling.
The first one keeps the patience-time distribution fixed in a heavy-traffic regime.  This stream can be further classified depending on the assumed heavy-traffic regime.
In the conventional heavy-traffic regime, \cite{WardGlynn2003a} identified the diffusion limit as a reflected Ornstein-Uhlenbeck process for the $M/M/1+M$ model.
Later, \cite{WardGlynn2005} extended the result to the general $G/GI/1+GI$ model.
In the Halfin-Whitt regime, \cite{GMR2002} obtained the diffusion limit for the $M/M/n+M$ model.
\cite{DaiHeTezcan2010} extended the diffusion analysis to a more general $G/PH/n+GI$ model by applying a general continuous map to both the fluid and diffusion-scaled processes and the random-time-change theorem.
\cite{MandelbaumMomcilovic2010} derived diffusion approximations for the $G/GI/n+GI$ queue building on the work on the $G/GI/n$ queue by \cite{Reed2009}.
In the non-degenerate slowdown regime (NDS), \cite{Atar2012NDS} established the diffusion approximation for the model with Poisson arrivals and exponential service and patience times.
Results of all the above studies share the common feature that only the density of the patience-time distribution at the  origin plays a role in the diffusion limit.

Based on a statistical study of call center data, however, \cite{ZeltynMandelbaum2005} pointed out that the estimate of the hazard-rate function of patience times at a single point often turns out to be unstable.
To preserve more information about the patience-time distribution, another stream of the literature scales the patience-time distribution by the hazard rate, rather than by the density at a single point.
\cite{ReedWard2008} obtained the diffusion approximations for both the offered waiting-time process and the queue length process for the $G/GI/1+GI$ model in the conventional heavy-traffic regime.
Their approach was to use a non-linear generalized regulator mapping to establish weak convergence results.
Taking advantage of  the memoryless property of exponential distributions, recently, \cite{ReedTezcan2012} applied the same hazard-rate scaling to study the diffusion limit of the queue length process for the $G/M/n+GI$ model, which was extended by \cite{Weerasinghe2014} to allow a state-dependent service rate.
\cite{Katsuda15}, again by taking advantage of the memoryless property, extended the service time to be phase-type and allowed patience times to be more general.
The extensive numerical experiments of \cite{ReedTezcan2012} showed that the approximations involving the entire hazard-rate function outperformed those that relied only on the density at the origin when the density of the patience-time distribution changes rapidly near the origin.
\begin{table}[h]
  \centering
  \begin{tabular}{|l|l|l|}
    \hline
    & No Scaling & With Scaling \\
    \hline
    \multirow{3}{*}{Conventional} & \cite{WardGlynn2003a} $M/M/1+M$ &  \\
    & \cite{WardGlynn2005} $G/GI/1+GI$ & \cite{ReedWard2008} $G/GI/1+GI$ \\
    & \cite{LeeWeerasinghe2011} & \cite{LeeWeerasinghe2011} \\
    & {\hskip 115pt} $G/GI/1+GI$ & {\hskip 110pt} $G/GI/1+GI$\\
    \hline
    NDS & \cite{Atar2012NDS} $M/M/n^\alpha+M$ &  \\
    \hline
    \multirow{4}{*}{Halfin-Whitt} & \cite{GMR2002} $M/M/n+M$  & \cite{ReedTezcan2012} $G/M/n+GI$ \\
    & \cite{DaiHeTezcan2010} $G/PH/n+GI$ & \cite{Weerasinghe2014} $G/M/n+GI$  \\
    & \cite{MandelbaumMomcilovic2010} & \cite{Katsuda15} $G/PH/n+GI$ \\
    & {\hskip 112pt} $G/GI/n+GI$ & \\
    \hline
  \end{tabular}
  \caption{Diffusion Approximations  for Systems with Abandonment}
  \label{tab:summary}
\end{table}
Table~\ref{tab:summary} summarizes the existing studies on the diffusion analysis of queueing systems by classifying them into three heavy traffic regimes and two scalings of the patience-time distribution.
Readers are referred to \cite{Ward2012} for a comprehensive survey on the study of customer abandonment both without scaling and with hazard-rate scaling of patience times.

Based on the intuition developed by \cite{ReedWard2008}, recently \cite{DaiHe2013} proposed a neat approximation for the scaled abandonment process when the service time distribution is generalized from exponential to phase-type.
The approximation is expressed as an integral whose integrand is just the hazard rate function and the integral limit is given by the diffusion approximation for the number of customers in the system.
Numerically, they showed that their approximation is remarkably accurate.
{\it But one would hope to see a rigorous proof of their proposed approximation for phase-type service times}.
{\it Furthermore, it would be interesting to build the diffusion approximation for $G/GI/n+GI$ with hazard rate scaling of the patience-time distribution.}

From the methodological perspective, the above-mentioned works are about different models and set in different heavy-traffic regimes  (see Table~\ref{tab:summary}).
The analysis for single-server queues in the conventional heavy-traffic regime and that for many-server queues in the Halfin-Whitt regime and the NDS regime require different methods.
For example, \cite{WardGlynn2005} used the virtual waiting time for the single-server queues while \cite{MandelbaumMomcilovic2010} relied on the analysis of the queue length process for $G/GI/n+G$ queues in the Halfin-Whitt regime;
in contrast with these two papers, however \cite{Atar2012NDS} directly constructed a Poisson process to represent the abandonment process by taking advantage of the memoryless property of the exponential patience-time distribution for $M/M/n^{\alpha}+M$ in the NDS regime.
Moreover, for patience time with and without scaling, the methods are quite different even in the same regime.
In the Halfin-Whitt regime, for instance, when considering $G/GI/n+G$ without scaling, \cite{MandelbaumMomcilovic2010} constructed an auxiliary system with which to analyze the queue length process of the original system;
while considering $G/M/n+G$ with scaling, \cite{ReedTezcan2012} and \cite{Weerasinghe2014} directly proved the asymptotic equivalence between the queue length process and the virtual waiting-time process   to obtain the diffusion limit  of the queue length process.
\emph{It would be nice to have a unified approach that applies across different regimes, and that can treat the patience time with or without scaling.}

Motivated by the above problems, our goal in this paper is to provide a uniform approach to the diffusion analysis of single-server queues and many-server queues with and without hazard-rate scaling.
%
The framework for modeling the patience-time distribution described in \eqref{eq:pat-distr-scale} can cover no-scaling, hazard-rate scaling and several other types of scalings,  which can potentially be used to analyze customer abandonment behaviors.
We focus on the unified approach in establishing the diffusion limits under the general scaling for the customer patience time \eqref{eq:pat-distr-scale}.
The approach has two steps.
The first step is to identify an asymptotic relationship between the customer abandonment process and the queue length process in Theorem~\ref{theo:abandoncouple} based on the general scaling \eqref{eq:pat-distr-scale} for the patience-time distribution.
When \eqref{eq:pat-distr-scale} is specialized to the case without scaling, our result reduces to that of \cite{DaiHe2010}.
Such an asymptotic relationship is established by using the patience-time distribution to connect the abandonment process to the virtual waiting time process, which can be approximated with the queue length process by proving a generalization of Little's law.
The challenge caused by the general scaling \eqref{eq:pat-distr-scale} is that the queue length processes are required to be tight, while only stochastic boundedness is
needed for the case without scaling as in \cite{DaiHe2010}.
Tightness, in particular the modulus of continuity \eqref{eq:modular-cont}, is usually difficult to verify.
To tackle this challenge,
we establish the tightness of the abandonment processes based only on the stochastic boundedness of the queue length processes, which forms a part of Theorem~\ref{theo:abandoncouple}.
Having tightness of the abandonment processes allows us to verify the tightness of
the queue length processes via the second step of our approach.
The second step is to construct a mapping which would reveal a functional relationship between the system status (such as the queue length process) and the stochastic primitives (such as the arrival process, service and patience times).
The mapping, with some nice properties, not only helps to verify the tightness of the queue length processes required by Theorem~\ref{theo:abandoncouple}, but also provides diffusion analysis by applying the continuous mapping theorem.
Within the unified framework described in the above two steps, to develop diffusion analysis for queueing systems with abandonment, it is enough to construct such continuous mappings and verify some mild assumptions.
Those assumptions can be verified in a same way as that for systems without abandonment.

We demonstrate how to use our approach to establish diffusion approximations via three examples, which are all new results in the literature.
In the first example (Section~\ref{sec:single-server}), we study the single service queue with Markov-modulated service speeds in the traditional heavy-traffic regime.
See \cite{MahabhashyamGautam2005} and \cite{Takine2005} for a wide range of applications of such models in telecommunications and web servers.
The classical single-server queue studied by \cite{WardGlynn2005} and \cite{ReedWard2008} can be viewed as a special case where the service speed is constant.
In the second example (Section~\ref{sec:Halfin-Whitt}), we establish the diffusion approximations for many-server queues in the Halfin-Whitt regime with general service times.
The special case of no scaling is the result in \cite{MandelbaumMomcilovic2010} and the special case with exponential service times and scaling is the result in \cite{ReedTezcan2012} and \cite{Weerasinghe2014}.
Moreover, the diffusion approximation established here justifies the approximation of the scaled abandonment processes proposed by \cite{DaiHe2013}.
In the third example (Section~\ref{sec:NDS}), we study the many-server queues in the NDS regime by extending the work of \cite{Atar2012NDS} to general patience-time distribution. 
These three examples shows that the advantage of our unified approach is to simplify the diffusion analysis of queues with customer abandonment by making it essentially the same as the diffusion analysis of queues without customer abandonment.

The rest of this paper is organized as follows.
We introduce our unified approach in Section~\ref{sec:fram-cust-aband}, but postpone the proof to Section \ref{sec:outline-proof}.
Section~\ref{section:diffusion-analysis} demonstrates the application of the unified approach.
In particular, we consider three systems: the $G/GI/1$+$GI$ queue with Markov-modulated service speeds in the conventional heavy-traffic regime in Section~\ref{sec:single-server}, the $G/GI/n$+$GI$ queue in the Halfin-Whitt regime in Section~\ref{sec:Halfin-Whitt}, and the $G/M/n^\alpha$+$GI$ queue with $\alpha\in (0, 1)$ in the NDS regime in Section~\ref{sec:NDS}.
Several technical proofs are presented in the Appendix.

Before we conclude this section, we introduce some notation and definitions that are used throughout the paper.
All random variables and processes are defined on a common probability space $(\Omega, {\mathcal {F}}, \mathbb{P})$ unless otherwise specified.
We denote by $\Z_+$, $\R$ and $\R_+$ the sets of positive integers, real numbers and nonnegative numbers, respectively.
The space of RCLL (right continuous with left-hand limits) functions on $\R_+$ taking values in $\R$ is denoted by $\D(\R_+,\R)$, and the subspace of the continuous functions in $\D(\R_+,\R)$ is denoted by $\C(\R_+,\R)$.
The space $\D(\R_+,\R)$ is assumed to be endowed with the Skorohod $J_1$-topology (see \cite{Billingsley1999}).
For $g\in \D(\R_+,\R)$, $g(t-)$ represents its left limit at $t>0$, and the uniform norm of $g(\cdot)$ on the interval $[a, b]$ is defined as
\begin{equation*}
\|g\|_{[a,b]}=\sup_{t\in [a,b]}|g(t)| \ \
\mbox{with }
\|g\|_{[0,b]} \textrm{ abbreviated to }\|g\|_b.
\end{equation*}
For a sequence of random elements $\{X^n, n\in \Z_+\}$ taking values in a metric space, we write $X^n \Rightarrow X$ to denote the convergence of $X^n$ to $X$ in distribution.
$X\stackrel{d}{=}Y$ means that random elements $X$ and $Y$ have the same distribution.
For $a\in \R$, $a^+ = \max\{a, 0\}$, $a^-=\max\{-a,0\}$, and $\lfloor a \rfloor$ is the largest integer not greater than $a$.
We use $\mathds{1}_{A}$ to denote the indicator function of set $A \subset \Omega$.

\section{Model and Asymptotic Framework}
\label{sec:fram-cust-aband}

Consider a sequence of first-come first-served (FCFS) $G/GI/N_n+GI$ queues indexed by $n\in\Z_+$, where $N_n$ is deterministic and represents the number of servers in the $n$th system.
Denote by $Q^n(t)$ the number of customers in the queue at time $t$, by $X^n(t)$ the total number of customers in the system at time $t$, and by $G^n(t)$ the number of customers who have abandoned the queue by time $t$, in the $n$th system.
In this paper, we assume, for technical convenience, that the patience times of the customers who are initially in the system are infinite, i.e., the initial customers in the queue are infinitely patient (this assumption is not restrictive; see \cite{MandelbaumMomcilovic2010} for the study on the many-server queue).
Clearly, $G^n(0)=0$ and $Q^n(0)$ is the number of customers waiting in the queue at time zero. Define the diffusion-scaled processes
$\tilde{Q}^n=\{\tilde{Q}^n(t): t\geq 0\}$,
$\tilde{X}^n=\{\tilde{X}^n(t): t\geq 0\}$, and
$\tilde{G}^n=\{\tilde{G}^n(t): t\geq 0\}$ as
\begin{equation}
  \label{eq:diffusion-def-gen}
  \tilde{Q}^n(t)=\frac{Q^n(t)}{\sqrt{n}},\quad
  \tilde{X}^n(t)=\frac{X^n(t)-N_n}{\sqrt{n}},\quad
  \tilde{G}^n(t)=\frac{G^n(t)}{\sqrt{n}}.
\end{equation}
Our objective in this section is to prove an  asymptotic relationship (Theorem~\ref{theo:abandoncouple}) between $\tilde{Q}^n$ and $\tilde{G}^n$ under appropriate assumptions.

Let $E^n(t)$ denote the number of arrivals by time $t$ in the $n$th system, and define the diffusion-scaled arrival process $\tilde{E}^n=\{\tilde{E}^n(t): t\geq 0\}$ by
\begin{equation*}
  \tilde{E}^n(t)=\frac{E^n(t)-\lambda^n t}{\sqrt{n}},
\end{equation*}
where $\lambda^n$ is called customer arrival rate for the $n$th system and satisfies
\begin{equation}
  \label{eq:lambda-limit}
  \lim_{n\rightarrow \infty} \frac{\lambda^n}{n}=\mu>0.
\end{equation}
We assume that
\begin{equation}\label{eq:arrival-diffusion}
\tilde{E}^n  \dto\tilde E
  \quad\textrm{as }n\to\infty,
\end{equation}
for some process $\tilde{E}=\{\tilde{E}(t): t\geq 0\}\in\C(\R_+, \R)$.
Here $\mu$ in \eqref{eq:lambda-limit} is usually related to customer service times.
The customer service times (characterized by customer service requirements and system service speed to process  the requirements) will be specified when a concrete system is investigated.
Let $\gamma^n_i$ be the patience time of the $i$th arriving customer in the $n$th system.
A customer waiting in the system will leave without receiving service once his patience time is exhausted.
$\{\gamma^n_i, i\in\Z_+\}$ is assumed to be a sequence of i.i.d. random variables, and independent of the arrival process $E^n$ for each $n$.
We denote the patience-time distribution by $F^n(\cdot)$ and assume that for each $x\geq 0$,
\begin{equation}\label{eq:pat-distr-scale}
  \sqrt{n}F^n(\frac{x}{\sqrt{n}})\rightarrow f(x),\quad \hbox{as}\ n\rightarrow\infty,
\end{equation}
where $f(\cdot)$ is nondecreasing.
We assume that $f(\cdot)$ is locally Lipschitz continuous function, i.e., for any $T\geq 0$, there is a constant $\Lambda_T$ such that for all $x, y\in [0, T]$,
\begin{equation}
  \label{eq:Lipschitz}
  |f(x)-f(y)|\leq \Lambda_{T} |x-y|.
\end{equation}
As pointed out in the introduction, not only can this framework cover the two well-known ways of scaling patience-time distributions, namely, no scaling and hazard-rate scaling, but also provides some new types of scalings:
\begin{itemize}
\item \emph{No scaling.}
  Let $F^n(x)=F(x)$ for $x\ge 0$, where $F(\cdot)$ is a probability distribution function with $F(0)=0$ and $F'(0+)=\alpha$.
  In this case, $f(x)=\alpha x$ for $x\ge 0$.
\item \emph{Hazard-rate scaling.}
  Let $F^n(x)=1-\exp(-\int_0^xh(\sqrt{n}s)ds)$ for $x\ge 0$, for some locally Lipschitz continuous hazard-rate function $h(\cdot)$.
  In this case, $f(x)=\int_0^xh(s)ds$ for $x\ge 0$.
\item \emph{Mixture of hazard-rate scaling and no scaling.}
  For any give $p\in(0,1)$, let $F(\cdot)$ be a distribution function and $h(\cdot)$ be a locally Lipschitz continuous hazard-rate function.
  Let $F^n(x)=pF(x)+(1-p)\big[1-\exp(- \int_0^{x}h(\sqrt{n}s)ds)\big]$, $x\ge 0$.
  In this case, $f(x)=pF'(0+)x+(1-p)\int_0^xh(s)ds$, $x\ge 0$.
\item \emph{Delayed hazard-rate scaling}.
  Let $h_1(\cdot)$ and $h_2(\cdot)$ be two locally Lipschitz continuous hazard-rate functions, and let
  \begin{eqnarray*}
  F^n(x)=\left\{
  \begin{array}{ll}
  1-\exp\Big(-\int^x_0 h_1(s)ds\Big), &\mbox{if $x\in [0, \frac{x_0}{\sqrt n}]$},\\[0.1in]
  1-\exp\Big(-\int_0^{x_0/\sqrt n} h_1(s)ds-\int_{x_0/\sqrt n}^x h_2(\sqrt n s)ds\Big), &\mbox{if $x\in (\frac{x_0}{\sqrt n}, \infty)$},
  \end{array}
  \right.
  \end{eqnarray*}
  where $x_0$ is a positive constant, is usually called delayed time point.
  Then
  \begin{equation*}
    f(x)=\left\{
      \begin{array}{ll}
        h_1(0)x,\ &\mbox{if $0\le x\leq x_0$};\\
        h_1(0)x_0+\int_{x_0}^xh_2(s)ds,\ &\ \mbox{if $x> x_0$}.
    \end{array}
  \right.
\end{equation*}
\end{itemize}
In order to obtain the asymptotic relationship (Theorem~\ref{theo:abandoncouple}), the key assumption is that the sequence of diffusion-scaled queue length processes  $\{\tilde Q^n,n\in\Z_+\}$ is $C$-tight.
That is, on any finite interval $[0, T]$, the sequence is stochastically bounded, i.e.,
\begin{equation}
  \label{eq:stoc-bdd}
  \lim_{\Gamma\rightarrow \infty} \limsup_{n\rightarrow \infty} \ \mathbb{P}\bigg\{ \sup_{0\leq t\leq T} \tilde Q^n(t) >\Gamma\bigg\}=0,
\end{equation}
and the modulus of continuity is asymptotically small, i.e., for any
$\varepsilon>0$,
\begin{equation}
  \label{eq:modular-cont}
  \lim_{\delta\rightarrow 0} \limsup_{n\rightarrow \infty}
  \prob\Big\{
  \sup_{{s,t\in [0,T]},{|s-t|<\delta}}
  |\tilde Q^n(s)-\tilde Q^n(t)|>\varepsilon
  \Big\}=0.
\end{equation}

\begin{theorem}\label{theo:abandoncouple}
If a sequence of $G/GI/N_n+GI$ queues satisfies
\eqref{eq:lambda-limit}--\eqref{eq:pat-distr-scale}
  and \eqref{eq:stoc-bdd}, then the sequence $\{\ds G, n\in \Z_+\}$ is
$C$-tight. Moreover, when \eqref{eq:Lipschitz} and
 \eqref{eq:modular-cont} also hold, we have that for each $T>0$,
  \begin{equation}
    \label{equation:abandoncouple}
    \sup_{0\leq t\leq T}\left|
      \tilde{G}^n(t)-\mu\int_0^tf(\frac{1}{\mu}\tilde{Q}^n(s)) ds
      \right|
      \Rightarrow 0,\quad \textrm{as}\ n\rightarrow\infty.
  \end{equation}
\end{theorem}

\begin{remark}\label{rem-bound}
  Note that $C$-tightness of $\{\ds G, n\in \Z_+\}$ implies that the fluid scaled process $({1}/{n})G^n(\cdot)$ converges to zero in probability, which is the fluid limit result for the abandonment process.
  Due to this result, the abandonment process is negligible in fluid scaling, hence analyzing the fluid limit of the system with abandonment is essentially the same as analyzing the fluid limit of the system without abandonment.
\end{remark}

Theorem~\ref{theo:abandoncouple} does not need any condition on the service process as long as the queue length processes satisfy {\rm \eqref{eq:stoc-bdd}--\eqref{eq:modular-cont}}.
Whether the patience times have hazard-rate scaling or no scaling,   the theorem yields the following result.

\begin{corollary}\label{corollary:abandon-limit}
  Assume that the sequence of $G/GI/N_n+GI$ queues satisfies \eqref{eq:lambda-limit}--\eqref{eq:modular-cont}.
  {\rm (i)} If the patience-time distribution has a hazard-rate scaling, namely, $F^n(x)=1-\exp(-\int_0^xh(\sqrt{n}s)ds)$ for some locally bounded hazard-rate function $h(\cdot)$, then
  \begin{equation}
    \label{equation:abandoncouple-limit-1}
    \sup_{0\leq t\leq T}\Big|
      \tilde{G}^n(t)- \int^t_0 \int^{\tilde Q^n(s)}_0 h\Big(\frac{u}{\mu}\Big) du ds
      \Big|
      \Rightarrow 0,\quad \textrm{as}\ n\rightarrow\infty;
  \end{equation}
  {\rm (ii)} If the patience-time distribution has no scaling, that is, $F^n(x)=F(x)$ with derivative $\alpha=F'(0+)$, then
  \begin{equation}
    \label{equation:abandoncouple-limit-2}
    \sup_{0\leq t\leq T}\Big|
      \tilde{G}^n(t)- \alpha\int^t_0 \tilde Q^n(s) ds
      \Big|
      \Rightarrow 0,\quad \textrm{as}\ n\rightarrow\infty.
  \end{equation}
\end{corollary}

Corollary \ref{corollary:abandon-limit} (ii) is the same as Theorem 2.1 of \cite{DaiHe2010}  who obtained such asymptotic relationship when the patience time is not scaled and only \eqref{eq:lambda-limit}--\eqref{eq:stoc-bdd} hold.
However, due to the general scaling \eqref{eq:pat-distr-scale} for patience-time distributions, we need the additional condition \eqref{eq:modular-cont} to deal with the nonlinearity of the function $f(\cdot)$.


The independence of specific queueing models for Theorem~\ref{theo:abandoncouple} enables us to develop a unified approach to diffusion analysis.
Note that among conditions required by Theorem~\ref{theo:abandoncouple}, \eqref{eq:lambda-limit}--\eqref{eq:Lipschitz} are
standard for the system parameters.
The applicability of Theorem~\ref{theo:abandoncouple} often depends on the verification of conditions \eqref{eq:stoc-bdd}--\eqref{eq:modular-cont}, in particular \eqref{eq:modular-cont}, which is often a major difficulty in most queueing analysis.
So we now provide a continuous mapping technique as the second step of our unified approach to overcome the difficulty associated with the verification of condition \eqref{eq:modular-cont} on the queue length processes.
This, consequently, leads to the diffusion approximations by the continuous mapping theorem.

To establish \eqref{eq:modular-cont} on the modulus of continuity for $\{\ds Q, n\in \Z_+\}$, in view of $\ds Q=(\ds X)^+$, it is sufficient to consider the modulus of continuity for $\{\ds X, n\in \Z_+\}$.
To this end, in view of Theorem \ref{theo:abandoncouple}, we define the centered abandonment process $\tilde G_c^n=\{\tilde G_c^n(t): t\geq 0\}$ as
\begin{equation}
  \label{eq:G-hat}
  \tilde G_c^n(t) = \ds G(t) -
  \mu\int_0^t f\Big(\frac{1}{\mu}(\ds X(s))^+\Big) ds.
\end{equation}
Suppose there exists a sequence of processes $\ds Y=\{\ds Y(t):t\geq 0\}$ and a mapping  $\Phi:\D(\R_+,\R)\to \D(\R_+,\R)$ such that
\begin{equation}
  \label{eq:reflection}
  \ds X = \Phi(\ds Y-\tilde G_c^n).
\end{equation}
Roughly speaking, $\ds Y$ is the centered diffusion scaled process of the system primitives.
Its exact form depends on the specific queueing system under examination.
See \eqref{eq:Y-single-mms}, \eqref{eq:Y-HW} and \eqref{2012-Aug-6-3} for the expressions of $\ds Y$ in the three concrete models studied in Section~\ref{section:diffusion-analysis}.
The following result characterizes the asymptotic behavior of $\tilde{X}^n$.

\begin{theorem}\label{thm:unified-reflection}
  Assume that
  {\rm(i)} condition \eqref{eq:stoc-bdd} on $\{\ds Q, n\in \Z_+\}$ holds$;$
  {\rm (ii)} there exists $\tilde Y\in \C(\R_+,\R)$ such that $\ds Y\dto \tilde Y$ as $n\to\infty$$;$
  {\rm (iii)} the mapping $\Phi(\cdot)$ is Lipschitz continuous in the topology of uniform convergence over bounded intervals, measurable with respect to the Borel $\sigma$-field generated by the Skorohod $J_1$-topology, and $\Phi(\C(\R_+,\R))\subseteq \C(\R_+,\R)$.
  Then
  \begin{equation}
    \label{eq:X-conv}
    \tilde{X}^n\Rightarrow \tilde{X}=\Phi(\tilde Y) \ \textrm{ as }\ n\to\infty.
  \end{equation}
\end{theorem}

The proofs of Theorems \ref{theo:abandoncouple}--\ref{thm:unified-reflection}, and Corollary
\ref{corollary:abandon-limit} are postponed to Section~\ref{sec:outline-proof}.
In their proofs, we can see that Theorem~\ref{theo:abandoncouple} is  a key step to prove Theorem~\ref{thm:unified-reflection}.
Theorem~\ref{thm:unified-reflection} outlines our unified approach in more detail.
We first obtain the stochastic boundedness for $\{\ds Q, n\in \Z_+\}$ by a comparison with the systems without customer abandonment.
Then we construct the continuous mapping $\Phi(\cdot)$.
The principle of the construction of $\Phi(\cdot)$ is to make the weak convergence of $\{\ds Y, n\in \Z_+\}$ to be tractable, which can usually be established by going along the same way as in the systems without abandonment.
Hence the approach developed here makes the diffusion analysis of queues with customer abandonment to be essentially same as the diffusion analysis of queues without customer abandonment.

Next we apply this approach to the diffusion analysis for the $G/GI/N_n+GI$ queue.

\section{Diffusion Analysis for $G/GI/N_n+GI$}
\label{section:diffusion-analysis}

The setup for the sequence of the $G/GI/N_n+GI$ systems is as follows.
For the $n$th system, let $v^n_i$, $i=1, 2, \ldots$ be the service requirement of the $i$th customer who arrives at the system after time 0 and will not abandon.
For $i=-X^n(0)+1, \cdots, -Q^n(0)$, $v^n_i$ denotes the remaining service requirement of the $i$th customer initially in service.
For $i=-Q^n(0)+1,\cdots, 0$, $v^n_i$ denotes the service requirement of the $i$th customer initially waiting in the queue.
Customer $-Q^n(0)+1$ is the first in the queue, customer $-Q^n(0)+2$ is the second, and so on.
We assume $\{v^n_i,i\geq -X^n(0)+1\}$ is a sequence of independent random variables, and is independent of the patience times $\{\gamma^n_i,i\in\Z_+\}$ and the arrival process $E^n$ given in Section \ref{sec:fram-cust-aband} for each $n$.
We assume the convergence of initial states,
\begin{equation}
  \label{eq:X-0-diffusion}
  \tilde{X}^n(0)\Rightarrow \xi \quad \textrm{as }n\to\infty,
\end{equation}
for some random variable $\xi$. We also assume the following \emph{heavy-traffic condition},
\begin{equation}
  \label{eq:HT-QED}
  \beta^n:=\sqrt{n}(\frac{\lambda^n}{n\mu}-1)
  \to \beta \quad \textrm{as }n\to\infty,
\end{equation}
for some $\beta\in\mathbb{R}$.
In particular, the heavy-traffic condition implies \eqref{eq:lambda-limit}.
Our diffusion approximation results will be established in the heavy-traffic regime specified by \eqref{eq:HT-QED} with assumption \eqref{eq:arrival-diffusion} on the arrival processes, assumptions \eqref{eq:pat-distr-scale}--\eqref{eq:Lipschitz} on the patience-time distribution, and the initial condition \eqref{eq:X-0-diffusion}.
The relationship between $\mu$ in the heavy-traffic condition \eqref{eq:HT-QED} and the means  of the customer service requirements $\{v^n_i,i\geq -X^n(0)+1\}$  will be characterized through the system service speed of processing the service requirements in the concrete models, see Assumptions~\ref{single}--\ref{assump:Atar}.

\subsection{$G/GI/1+GI$ in the Traditional Heavy-Traffic Regime}
\label{sec:single-server}

In this section, we study a sequence of single-server queues in the
traditional heavy-traffic regime. We adopt a general model to allow
the service speed in the $n$th system to be modulated by a
continuous-time Markov chain $\Delta^n=\{\Delta^n(t): t\geq 0\}$
with a finite state space ${\cal S}=\{1,\cdots, \ell\}$. At time
$t$, the server will process customer service requirements at speed
$n\mu_i$ when $\Delta^n(t)=i\in {\cal S}$. This is  a general model
as the classical single-server queue is a special case where the
state space has only a single state, i.e., the service speed is
constant. The following setup for the model is standard (see
\cite{DVZ2013}).

\begin{assumption}\label{single}
  The customer service requirements $\{v^n_i, i\geq -X^n(0)+1\}$ are independent and identically distributed with mean $1$ and variance $\theta^2$.
  The Markov chain $\{\Delta^n(t): t\geq 0\}$ is given by $\Delta^n(t)=\Delta(nt)$ for $t\geq 0$, where $\Delta=\{\Delta(t): t\geq 0\}$ is an irreducible and stationary continuous-time Markov chain with state space ${\cal S}$, generator ${\cal G}$ $((\ell\times\ell)$-matrix$)$, and stationary distribution $\pi:=(\pi_1,\cdots,\pi_\ell)$.
\end{assumption}

We assume $\mu$ in \eqref{eq:lambda-limit} is equal to $\sum_{i\in
{\cal S}}\pi_i\mu_i$, which can be considered as the long-run
average speed at which the server processes service requests. The
following preliminary result on continuous Markov chains will be
needed in establishing the diffusion approximation for the
$G/GI/1+GI$ with Markov-modulated service speeds. For its proof, see
\cite{YinZhang2013}.
\begin{lemma}\label{lem:Mark}
  Under Assumption {\rm\ref{single}}, for any $T\geq 0$, as $n\to\infty$,
  \begin{eqnarray*}
    \sup_{0\leq t\leq T}\Big|\int_0^t\mu_{\Delta^n(s)} ds- \sum_{i\in {\cal S}} \mu_i\pi_i t
    \Big| \Rightarrow 0  \ \ \mbox{and} \ \ \tilde \Delta^n \Rightarrow \tilde \Delta,
  \end{eqnarray*}
  where $\tilde \Delta^n=\{\tilde \Delta^n(t): t\geq 0\}$ given by
  \begin{eqnarray*}
    \tilde \Delta^n(t)= \sqrt n \Big(
    \int_0^t\mu_{\Delta^n(s)} ds-
    \sum_{i\in {\cal S}} \mu_i\pi_i t\Big) ,
  \end{eqnarray*}
  and $\tilde{\Delta}=\{\tilde \Delta(t): t\geq 0\}$ is a Brownian motion with zero drift, and variance $\theta^2_{\cal S}$ given by
  \[
  \theta_{\cal S}^2=\sum_{i,j\in {\cal S}}\mu_i\mu_j \Big(\pi_i\int^\infty_0 \varphi_{ij}(s)ds+\pi_j\int^\infty_0
  \varphi_{ji}(s)ds\Big)
  \]
  with $(\varphi_{ij}(s))_{\ell\times \ell}=(I-(1,\cdots, 1)^\prime \cdot \pi)\cdot \exp({\cal G}s)$, where $I$ is an $\ell$-by-$\ell$ identity matrix.
\end{lemma}

The total amount of customer service requests processed by the
server during $[0, t)$ is $\int_0^t n\mu_{\Delta^n(s)}\cdot
(X^n(s)\wedge 1) ds$.
Define
\begin{equation*}
  S^n(t)=\max\left\{k: v^n_{-X^n(0)+1}+\cdots + v^n_{-X^n(0)+k}\leq t\right \}
\end{equation*}
as the renewal process associated with the sequence of service requirements. Then the number of customers served by time $t$ is
\begin{equation*}
  S^n\Big(n\int_0^t \mu_{\Delta^n(s)}\cdot ( X^n(s)\wedge 1) ds\Big).
\end{equation*}
The evolution of the process $X^n$ can be characterized by the
system dynamics equation
\begin{equation}
  \label{equ:sysdyn}
  X^n(t)=X^n(0)+E^n(t)-S^n\Big(n\int_0^t\mu_{\Delta^n(s)}\cdot ( X^n(s)\wedge 1) ds\Big)-G^n(t).
\end{equation}
In view of \eqref{eq:diffusion-def-gen} and \eqref{eq:G-hat}, we rewrite
\eqref{equ:sysdyn} as
\begin{equation}
  \label{eqn:diffdyn}
  \begin{split}
    &\tilde X^n(t)  =
    \tilde Y^n(t)-\tilde{G}^n_c(t)
    -\mu\int_0^t f\Big(\frac{1}{\mu}(\tilde X^n(s))^+\Big)ds
    +n\int_0^t\mu_{\Delta^n(s)}(\tilde X^n(s))^- ds,
  \end{split}
\end{equation}
where
\begin{eqnarray}
  \tilde Y^n(t)&= & \tilde X^n(0)
   +\tilde E^n(t)-\tilde S^n\Big(\int_0^t\mu_{\Delta^n(s)}\cdot (X^n(s)\wedge 1) ds\Big)
 +\mu\sqrt n\Big(\frac{\lambda^n}{n\mu} -1\Big) t- \tilde \Delta^n(t),\label{eq:Y-single-mms}\\
  \tilde S^n(t)&=&\frac{S^n(nt)-nt}{\sqrt n}.\nonumber
\end{eqnarray}
In order to use Theorem~\ref{thm:unified-reflection}, we first
establish the following lemma.

\begin{lemma}\label{lem:mapping-abd-single}
Assume that  $g(\cdot)$ is a locally Lipschitz continuous function
defined on $\R_+$ with $g(0)=0$.  For any $y(\cdot)\in \D(\R_+,\R)$,
there exists a unique solution $(x(\cdot), z(\cdot))$ to the
following set of equations
\begin{eqnarray}
  &&x(t)= y(t)+\int_0^tg((x(s))^+)ds +z (t),\label{2013-Nov-1}\\
  && x(t)\geq 0,\nonumber\\
  && z(\cdot) \mbox{ is nondecreasing and }z(0)=0,\nonumber\\
  && \int_0^{\infty}x(s)dz(s)=0.\nonumber
\end{eqnarray}
Moreover, the mapping $\Phi_{g}(\cdot):\D(\R_+,\R)\rightarrow
\D(\R_+,\R)$ defined by $x=\Phi_{g}(y)$ is Lipschitz continuous in
the topology of uniform convergence over bounded intervals,
measurable with respect to the Borel $\sigma$-field generated by the
Skorohod $J_1$-topology, and $\Phi_{g}(\C(\R_+,\R))\subseteq
\C(\R_+,\R)$.
\end{lemma}

It is worth pointing out that the mapping is continuous in the Skorohod $J_1$-topology according to Proposition~4.9 of \cite{LeeWeerasinghe2011}.
Their proof is based on the earlier work of \cite{ReedWard2008} and some classical results of the Skorohod $J_1$-topology.
In Appendix~\ref{mapping}, we provide a simple and direct way to show the Lipschitz continuity under the uniform topology and demonstrate that this is sufficient for our reflection mapping approach.

\begin{theorem}
  \label{theorem:diffusion-single}
  Assume that conditions~\eqref{eq:arrival-diffusion}--\eqref{eq:Lipschitz}, and \eqref{eq:X-0-diffusion}--\eqref{eq:HT-QED} hold.
  For the stochastic processes $\{\tilde X^n, n\in \Z_+\}$ associated with the sequence of $G/GI/1+GI$ systems, if Assumption~{\rm \ref{single}} holds, then $\tilde{X}^n\Rightarrow \tilde{X}$ as $n\to\infty$.
  Here $\tilde{X}=\{\tilde X(t): t\geq 0\}$ is given by $\tilde{X}=\Phi_g(\tilde Y)$ with
  \begin{align*}
    g(t) &=-\mu f\Big(\frac{1}{\mu}t \Big),  \ t\geq 0;\\
    \tilde Y &=\{\tilde Y(t): t\geq 0\}, \
    \tilde Y(t)=\xi +\tilde{E}(t) - \sqrt{\mu}\theta\tilde{S}(t) - \tilde \Delta(t)+\beta\mu t,
  \end{align*}
  where $\tilde S=\{\tilde{S}(t): t\geq 0\}$ is a standard Brownian motion independent of $\xi$, $\tilde{E}$ and $\tilde{\Delta}$.
  Moreover, $\ds Q \dto \tilde X$ as $n\to\infty$.
\end{theorem}

\begin{proof}
  By (\ref{eqn:diffdyn}), we have
  \begin{equation}
    \label{eqn:diffdyn-1}
    \begin{split}
      &(\tilde X^n(t))^+  =
      \tilde Y^n(t)+ (\tilde X^n(t))^- -\tilde{G}^n(t)
      + n\int_0^t\mu_{\Delta^n(s)}(\tilde X^n(s))^- ds.
    \end{split}
  \end{equation}
  Note that
  \begin{equation*}
    \sup_{0\leq t \leq T}
    \Big|\tilde S^n\Big (\int_0^t\mu_{\Delta^n(s)}\cdot (X^n(s)\wedge 1) ds\Big) \Big|
    \leq
    \sup_{0\leq t \leq \max_{i\in {\cal S}}\{\mu_iT\}}
    \Big|\tilde S^n (t)\Big|.
  \end{equation*}
  By \eqref{eq:arrival-diffusion}, \eqref{eq:X-0-diffusion}--\eqref{eq:HT-QED}, Lemma \ref{lem:Mark} and the definition of $\tilde Y^n$ in \eqref{eq:Y-single-mms}, $\{\tilde Y^n, n\in \Z_+\}$ is stochastically bounded.
  It is clear, by $(\tilde X^n(t))^-\leq 1/\sqrt n$, that $\{(\tilde X^n)^-, n\in \Z_+ \}$ is also stochastically bounded.
  In view of $\tilde Y^n(t)+(\tilde X^n(t))^- -\tilde{G}^n(t) \leq \tilde Y^n(t)+ (\tilde X^n(t))^-$ and \eqref{eqn:diffdyn-1}, it follows from Lemma 4.1 in \cite{KLRS2007} that $(\tilde X^n)^+$ can be bounded by the one-dimensional Skorohod mapping of $\tilde Y^n+ (\tilde X^n)^-$.
  Hence, $\{(\tilde X^n)^+, n\in \Z_+\}$ is stochastically bounded.
  This consequently implies that for any $T\ge 0$, as $n\to\infty$,
  \begin{equation}
    \label{eqn:fluid-queue}
    \sup_{0\leq t\leq T}\frac{(X^n(t)-1)^+}{n} \Rightarrow 0.
  \end{equation}
  We now prove the convergence of the sequence $\{\tilde Y^n, n\in \Z_+\}$.
  It follows from the above stochastic boundedness analysis that for any $T\ge 0$, as $n\to\infty$,
  \begin{equation}
    \label{eqn:y-term}
    \sup_{0\leq t\leq T} \frac{\tilde Y^n(t)+(\tilde X^n(t))^--\tilde G^n(t)}{\sqrt n}
    \Rightarrow 0.
  \end{equation}
  By \eqref{eqn:diffdyn-1}--\eqref{eqn:y-term}, as $n\to \infty$,
  \begin{eqnarray*}
      && \sup_{0\leq t\leq T}\int_0^t \mu_{\Delta^n(s)}(X^n(s)-1)^-ds \Rightarrow 0.
  \end{eqnarray*}
  The above limit together with  Lemma \ref{lem:Mark} implies that as $n\to\infty$,
  \begin{equation}
    \label{eq:mms-fluid-rate}
    \sup_{0\leq t\leq T}\left|\int_0^t\mu_{\Delta^n(s)}\cdot ( X^n(s)\wedge 1) ds- \mu t\right|\Rightarrow0.
  \end{equation}
  Applying the above limit and the random-time-change theorem (Corollary~1 of \cite{Whitt1980}), the sequence of processes $\{\tilde S^n (\int_0^t\mu_{\Delta^n(s)}\cdot (X^n(s)\wedge 1) ds ): t\geq 0\}$ converges in distribution to
  $\sqrt{\mu}\theta \tilde S$  with $\{\tilde S(t): t\geq 0\}$ being a Brownian motion.
  It follows from conditions \eqref{eq:arrival-diffusion}, \eqref{eq:X-0-diffusion}-\eqref{eq:HT-QED} and Lemma~\ref{lem:Mark} that
  \begin{equation}
    \tilde Y^n\Rightarrow \tilde Y,
  \end{equation}
  where $\tilde Y=\{\tilde Y(t): t\geq 0\}$ with $\tilde Y(t)=\xi +\tilde{E}(t)-\sqrt{\mu}\theta\tilde{S}(t)-\tilde{\Delta}(t)+ \beta\mu t$.
  We have so far verified conditions (i) and (ii) in Theorem~\ref{thm:unified-reflection}. Condition~(iii) follows directly from \eqref{eqn:diffdyn} and Lemma~\ref{lem:mapping-abd-single}.
  This completes the proof.
\end{proof}

\begin{remark}
  For the classical $G/GI/1+GI$, the limit $\tilde \Delta$ in Lemma~{\rm\ref{lem:Mark}} becomes $0$ since $\cal S$ contains only a single state.
  Theorem~{\rm \ref{theorem:diffusion-single}}, consequently, gives the weak convergence for the queue length process of the classical $G/GI/1+GI$ considered by {\rm    \cite{WardGlynn2005}} and {\rm \cite{ReedWard2008}}.
\end{remark}

\subsection{$G/GI/n+GI$ in the Halfin-Whitt Regime}
\label{sec:Halfin-Whitt}

In this subsection, we apply our unified approach to establishing the diffusion approximation for $G/GI/n+GI$ where customer service requests are assumed to be processed by each server at speed 1 without loss of generality.
Since we use the general scaling \eqref{eq:pat-distr-scale}, our result covers the case in \cite{MandelbaumMomcilovic2010} where patience times are not scaled, and the case in \cite{ReedTezcan2012} where hazard-rate scaling is applied to the patience times.
We also generalize the latter work to a generally distributed service requirement.

Let $H_e(\cdot)$ denote the equilibrium distribution associated with
the distribution $H(\cdot)$ of customer service requirements, i.e.,
\begin{equation*}
  H_e(x)=\mu\int_0^x(1-H(s))ds,\ x\geq 0.
\end{equation*}
Thus the renewal function of the delayed renewal process with initial distribution $H_e(\cdot)$ and inter-renewal distribution $H(\cdot)$ is $\mu t$.
The following assumption on the service process is required for this example.

\begin{assumption}\label{QED}
  The customer service requirements $\{v_i^n, i\geq -Q^n(0)+1\}$ are independent and identically distributed with distribution function $H(\cdot)$ which has mean $1/\mu$ and variance $ \theta^2$.
  The remaining service requirements of the customers who are initially in service, $\{v^n_i, -X^n(0)+1 \leq i \leq -Q^n(0)\}$, are independent and identically distributed with distribution function $H_e(\cdot)$.
  Moreover, the two sequences are independent.
\end{assumption}


Let $D^n(t)$ be the number of customers whose service requirements
have been completed by time $t$ in the $n$th system. We then have
the following simple balance equation for the total number of customers in the $n$th system at time $t$:
\begin{equation}
  \label{eq:balance}
  X^n(t)=X^n(0)+E^n(t)-D^n(t)-G^n(t).
\end{equation}
Let $M(\cdot)$ denote the renewal function associated with $\{v_i^n,
i\geq -Q^n(0)+1\}$, i.e., $M(\cdot)$ satisfies the following renewal
equation
\begin{equation}
  \label{eq:renewal-eq}
  M(t)=H(t)+\int_0^tH(t-s)d M(s).
\end{equation}
Define
\begin{align}
  D^n_c(t) &= D^n(t)-n\mu t -(X^n(0)-n)^-\cdot (M(t)-\mu t)
  + \int_0^t(X^n(t-s)-n)^-d M(s),\label{service}\\
  \tilde D^n_c(t) &=\frac{D^n_c(t)}{\sqrt n}. \nonumber
\end{align}
The idea of \eqref{service}, which follows Equation (33) in \cite{ReedShaki2014}, is to center the service completion process using the renewal function $M(\cdot)$.
Then \eqref{eq:balance} becomes
\begin{equation}
  \label{eq:balance-transformed}
  \begin{split}
  X^n(t)
  &=X^n(0)+E^n(t) -G^n(t) - D_c^n(t)\\
  &\quad - n \mu t - (X^n(0)-n)^- \cdot (M(t) -\mu t)
  + \int_0^t(X^n(t-s) - n)^-d M(s).
  \end{split}
\end{equation}
Applying diffusion scaling \eqref{eq:diffusion-def-gen} to
\eqref{eq:balance-transformed} implies that
\begin{equation}
  \label{eq:balance-diffusion}
  \begin{split}
  \tilde X^n(t) = \tilde Y^n(t) -\tilde G_c^n(t)
                + \int_0^t (\tilde X^n(t-s))^-d M(s)
                - \mu\int_0^tf(\frac{1}{\mu}(\ds X(s))^+) ds,
  \end{split}
\end{equation}
where $\tilde G_c^n=\{\tilde G_c^n(t):t\ge 0\}$ is defined as in
\eqref{eq:G-hat} and
\begin{equation}
  \label{eq:Y-HW}
  \tilde Y^n(t)
  = \tilde X^n(0)+\tilde E^n(t) - \tilde D^n_c(t) +\beta^n\mu t
  + (\tilde X^n(0))^-\cdot (\mu t-M(t)).
\end{equation}
The following proposition yields the weak convergence for $\{\tilde Y^n, n\in \Z_+\}$.

\begin{proposition}\label{propo:servicediffusion}
  Assume that conditions~\eqref{eq:arrival-diffusion}--\eqref{eq:Lipschitz}, \eqref{eq:X-0-diffusion}--\eqref{eq:HT-QED}, and Assumption~{\rm\ref{QED}} hold. For the sequence of $G/GI/n+GI$ systems,  $\tilde Y^n \Rightarrow \tilde Y$ with
  \begin{equation*}
    \tilde Y(t) = \xi + \tilde{E}(t) - \tilde{D}(t) + \beta\mu t + \xi^-\cdot \left(\mu t-M(t)\right),
  \end{equation*}
  where $\tilde{D}=\{\tilde{D}(t): t\geq 0\}$ is a zero-mean Gaussian process, which is independent of $\tilde{E}$ and $\xi$, with the covariance given by
  \begin{equation}\label{eq:S-co-var}
    \E\big[\tilde{D}(s)\tilde{D}(t)\big]
    = 2\int_0^s \left(M(u)-u+\frac{1}{2}\right)du
    +\int_0^s\int_0^{t} M(s-u)M(t-v) dH(u+v)
  \end{equation}
  for any $0\le s\le t$.
\end{proposition}

The proof of this proposition is postponed until after we have
established the diffusion approximation
Theorem~\ref{theorem:diffusion-multi}. In order to use
Theorem~\ref{thm:unified-reflection}, we introduce a regulator
mapping in the following lemma.

\begin{lemma}\label{lem:mapping-abd-HW}
Assume that  $g(\cdot)$ is a locally Lipschitz continuous function with $g(0)=0$.
For any $y(\cdot)\in \D(\R_+,\R)$ and $M(\cdot)$ given by \eqref{eq:renewal-eq}, there exists a unique solution $x(\cdot)$ to the following equation
\begin{equation}\label{eq:mapping-general}
  x(t)
  =y(t)+\int_0^t \left(x(t-s)\right)^-dM(s)
  +\int_0^t g(\left(x(s)\right)^+)ds.
\end{equation}
Moreover, the mapping $\Phi_{M,g}(\cdot):\D(\R_+,\R)\rightarrow
\D(\R_+,\R)$ defined by $x=\Phi_{M,g}(y)$ is Lipschitz continuous in
the topology of uniform convergence over bounded intervals,
measurable with respect to the Borel $\sigma$-field generated by the
Skorohod $J_1$-topology, and $\Phi_{M,g}(\C(\R_+,\R))\subseteq
\C(\R_+,\R)$.
\end{lemma}

This lemma is a generalization of Proposition~7 in \cite{Reed2007} in which $g(\cdot)\equiv 0$ is assumed.
The proof of this lemma is presented in Appendix~\ref{mapping}.
Following this lemma and \eqref{eq:balance-diffusion}, we have $\tilde X^n=\Phi_{M,g}(\tilde Y^n-\tilde G^n_c)$ with $g(t)=-\mu f(t/\mu)$.
Theorem~\ref{thm:unified-reflection} can now be applied to obtain the following diffusion approximation.

\begin{theorem}
  \label{theorem:diffusion-multi}
  Assume that conditions~\eqref{eq:arrival-diffusion}--\eqref{eq:Lipschitz} and \eqref{eq:X-0-diffusion}--\eqref{eq:HT-QED} hold. For the stochastic processes $\{\tilde X^n, n\in \Z_+\}$ associated with the sequence of $G/GI/n+GI$ systems, if Assumption~{\rm\ref{QED}} holds, then $\tilde{X}^n\Rightarrow \tilde{X}$ as $n\to\infty$, where $\tilde{X}=\{\tilde X(t):t\ge 0\}$ is the solution to the following
  \begin{equation}\label{eq:tilde X}
    \begin{split}
      \tilde{X}(t)=\tilde Y(t)+\int_0^t(\tilde{X}(t-s))^-dM(s)
      - \mu\int_0^tf(\frac{1}{\mu}(\tilde{X}(s))^+)ds,
    \end{split}
  \end{equation}
  and  $\tilde{Y}$ is given by Proposition {\rm \ref{propo:servicediffusion}}.
  Moreover, $\ds Q\dto \tilde X^+$ as $n\to\infty$.
\end{theorem}

\begin{proof} 
In view of Lemma~\ref{lem:mapping-abd-HW},
Proposition~\ref{propo:servicediffusion} and (\ref{eq:Y-HW}), we
just need to verify Theorem~\ref{thm:unified-reflection} (i). Let
$\tilde{Q}_0^n$ denote the queue length process of the many-server
system without abandonment. It is proved in \cite{Reed2009} that
$\{\tilde Q_0^n, n\in \Z_+\}$ is stochastically bounded. Again, by
Theorem~2.2  of \cite{DaiHe2010}, with probability one, $\ds Q(t)\le
\ds Q_0(t)$ for all $t\ge 0$. This implies that $\{\tilde{Q}^n, n\in
\Z_+\}$ is stochastically bounded.
\end{proof}


\begin{remark}
 For a given $n$-server system with patience-time distribution $F(\cdot)$, from  Theorem {\rm \ref{theorem:diffusion-multi}} and Corollary {\rm \ref{corollary:abandon-limit}}, we  can use
\begin{equation}
 \mu \int_0^t \sqrt{n}F\left(\frac{1}{\mu\sqrt{n}}\frac{Q(s)}{\sqrt{n}}\right)ds
\end{equation}
to approximate $G(t)/{\sqrt{n}}$.  In particular, if
$F(x)=1-\exp(-\int_0^x h(s)ds)$, then
\begin{equation}
\begin{split}
    \sqrt{n} \Big[1-\exp \Big(-\int_0^{ x/(\mu\sqrt{n})} h(s)ds \Big)\Big]=&\sqrt{n} \Big[1-\exp \Big(-\frac{1}{\mu\sqrt{n}}\int_0^{x} h \Big(\frac{1}{\mu\sqrt{n}}s \Big)ds \Big) \Big]\\
    \approx&\frac{1}{\mu}\int_0^{x} h\Big(\frac{1}{\mu\sqrt{n}}s \Big)ds\approx\frac{1}{\mu}\int_0^{x} h \Big(\frac{\sqrt{n}}{\lambda^n}s \Big)ds,
\end{split}
\end{equation}
which implies that $\int^t_0 \int^{{Q(s)}/{\sqrt{n}}}_0 h\Big(
\frac{\sqrt n u}{\lambda^n}\Big) du ds$ can  approximate
$G(t)/{\sqrt{n}}$ well.
{\rm\cite{DaiHe2013}} proposed 
this approximation for the scaled abandonment process $ G(t)/\sqrt
n$ when the patience-time distribution
$F(x)=1-\exp(-\int_0^xh(s)ds)$. 
Numerical experiments showed that their approximations are very
accurate. Hence, our Corollary {\rm \ref{corollary:abandon-limit}}
and Theorem {\rm \ref{theorem:diffusion-multi}} theoretically
validate their approximations from the perspective of the diffusion
approximations.
\end{remark}

\begin{remark}\label{special}
When $f(x)=\alpha x$, that is, there is no hazard rate scaling of
the patience-time distribution, Theorem {\rm
\ref{theorem:diffusion-multi}} gives the diffusion approximation of
the queue length for $G/GI/n+G$, which is obtained by {\rm
\cite{MandelbaumMomcilovic2010}}. If the service times are
independent and exponentially distributed $($Assumption {\rm
\ref{QED}} holds by the memoryless property of the exponential
distributions$)$, then Theorem {\rm \ref{theorem:diffusion-multi}}
gives the diffusion approximations for $G/M/n+G$ with the hazard
rate scaling, which is studied by  {\rm \cite{ReedTezcan2012}}.
\end{remark}

In order to obtain Proposition \ref{propo:servicediffusion}, we
introduce the following lemma that is related to the weak
convergence of the pure empirical processes, and is of independent
interest itself. Its proof is presented in Appendix~\ref{mapping}.
To describe the lemma, let $C^n=\{C^n(t): t\geq 0\}$ be a sequence
of counting processes and $\tau^n_i$ be its $i$th jump point.
Furthermore, for each $n\in \Z_+$, let $\{u^n_i, i\in \Z_+\}$ be a
sequence of i.i.d. random variables with a finite mean and some
distribution function $H_{\star}(\cdot)$. Define $\tilde {\cal
T}^n=\{\tilde {\cal T}^n(t): t\geq 0\}$ and $\tilde U^n=\{\tilde
U^n(t): t\geq 0\}$ with
\begin{eqnarray*}
  \tilde {\cal T}^n(t) &=& \frac{1}{\sqrt n}
  \sum_{i=1}^{C^n(t)}\Big(
    \id{\tau_i^n+u_i^n>t}-(1-H_{\star}(t-\tau^n_i))
  \Big),\\
  \tilde U^n(t) &=& \frac{1}{\sqrt n}
  \sum_{i=1}^{\lfloor n\mu t \rfloor}\Big(
    \id{\frac{i}{n\mu}+u_i^n>t}-(1-H_{\star}(t-\frac{i}{n\mu}))
  \Big).
\end{eqnarray*}

\begin{lemma}\label{lem:equivalent}
Assume that for each $k\in \Z_+$, $\{\tau^n_1,\cdots, \tau^n_k\}$
and $\{u^n_{i}, i\geq k\}$ are independent, and as $n\to\infty$
\begin{equation}
  \label{eq:C-conv-e}
  \bar C^n \Rightarrow \bar e,
\end{equation}
where $\bar e(t)=\mu t$.
Then for any $T>0$,
\begin{eqnarray}
\sup_{0\leq t\leq T}  \Big| \tilde {\cal T}^n(t)-\tilde
U^n(t)\Big|\Rightarrow 0, \label{2013-Dec-1}
\end{eqnarray}
and $\tilde {\cal T}^n \Rightarrow \tilde {\cal T}$ where $\tilde
{\cal T}=\{\tilde {\cal T}(t): t\geq 0\}$ is a Gaussian process
with continuous sample  paths, zero mean and covariance function
given by
  \begin{equation}
    \label{eq:M2-covarance}
    \E [\tilde {\cal T}(s) \tilde {\cal T}(t)]
    = \mu\int_0^sH_{\star}(s-u)[1-H_{\star}(t-u)] du, \ \ 0\leq s \leq t.
  \end{equation}
\end{lemma}

\begin{proof}[Proof of Proposition \ref{propo:servicediffusion}]
 The asymptotic analysis, in particular that of $\{D^n(t): t\geq 0\}$, follows the idea 
of \cite{Reed2009} and \cite{KrichaginaPuhalskii1997}. For
completeness, we include the proof here.

Let $K^n(t)$ be the number of customers who have entered service by
time $t$, and denote by $\kappa_i^n$ the $i$th jump time of the
counting process $\{K^n(t): t\geq 0\}$. Define
\begin{eqnarray*}
  \tilde M^n_1(t)&=&\frac{1}{\sqrt n}\sum_{i=-Q^n(0)+1}^{K^n(t)-Q^n(0)}\left(\id{\kappa_i^n+v_i^n>t}-(1-H(t-\kappa_i^n))\right), \\
 \tilde N^n_1(t)&=&\frac{1}{\sqrt n}
  \sum_{i=1}^{\lfloor\mu nt \rfloor}\Big(\id{\frac{i}{n\mu}+v_{-Q^n(0)+i}^n>t}-(1-H(t-\frac{i}{n\mu}))\Big), \label{eq:M2}
  \end{eqnarray*}
  and
  \begin{eqnarray*}
 \tilde M^n_0(t)&=&\frac{1}{\sqrt n}\sum_{i=-X^n(0)+1}^{-Q^n(0)}\left(\id{v_i^n>t}-(1-H_e(t))\right), \\
 \tilde  N^n_0(t)&=&\frac{1}{\sqrt n}\sum_{i=1}^n
  \left(\id{v_{-Q^n(0)-(n-i)}^n>t}-(1-H_e(t))\right), \label{eq:W0}
\end{eqnarray*}
where $\{v^n_i, i\leq -Q^n(0)\}$ are independent and identically distributed with  distribution function $H_e(\cdot)$, and independent of $\{v^n_i, i\geq -Q^n(0)+1\}$.
Hence, $\tilde N^n_1=\{\tilde N^n_1(t) :t\geq 0\}$  and $\tilde N^n_0=\{\tilde N^n_0(t) : t\geq 0\}$  are two independent processes.
Let $\tilde M^n_1=\{\tilde M^n_1(t) :t\geq 0\}$.
Similarly, we define the process  $\tilde M^n_0$.
By the weak convergence of the empirical processes (see Chapter 3 in
\cite{ShorackWellner2009}),  we have
\begin{eqnarray}
\tilde N^n_1 \Rightarrow  \tilde N_1 \ \ \mbox{and} \ \ \tilde N^n_0
\Rightarrow\tilde N_0, \label{2014-march-24-10}
\end{eqnarray}
where, by the independence of $\tilde N^n_1$ and $\tilde N^n_0$, $\tilde N_1=\{\tilde N_1(t): t\geq 0\}$ and $\tilde
N_0=\{\tilde N_0(t): t\geq 0\}$ are two independent Gaussian
processes with continuous sample  paths, zero mean  and covariance
functions given by  \eqref{eq:M2-covarance}, and  $H_e(s)\wedge
H_e(t)-H_e(s)H_e(t)$, respectively.
By the assumption on the independence between the service times 
and arrival process $\{E^n(t) : t\geq 0\}$, 
and in view of Theorem 2.8 in \cite{Billingsley1999}, we have
\begin{equation*}
  (\tilde E^n, \tilde N^n_1,\tilde N^n_0)\Rightarrow (  \tilde E,
  \tilde N_1, \tilde N_0).
\end{equation*}
Using (\ref{eq:X-0-diffusion}) and Theorem 3.9 in
\cite{Billingsley1999}, we obtain
\begin{eqnarray}
(\tilde X^n(0), \tilde E^n, \tilde N^n_1,\tilde N^n_0)\Rightarrow
(\xi, \tilde E, \tilde N_1, \tilde N_0). \label{2013-dec-21}
\end{eqnarray}
Again by \eqref{eq:X-0-diffusion}, we have
\[
\frac{X^n(0)}{n}\Rightarrow 1.
\]
In view of the definitions of $\{\tilde M^n_0(t): t\geq 0\}$ and
$\{\tilde N^n_0(t): t\geq 0\}$,
\begin{eqnarray}
\sup_{0\leq t\leq T}\Big|\tilde M^n_0(t)-\tilde
N^n_0(t)\Big|\Rightarrow 0. \label{2014-march-24-14}
\end{eqnarray}
Let $\bar K^n=\{K^n(t)/n: t\geq 0\}$.
Recall that in the proof of Theorem~\ref{theorem:diffusion-multi} of \cite{Reed2009} and
Theorem~2.2 of \cite{DaiHe2010}, the sequence of queue length processes $\{\tilde{Q}^n, n\in \Z_+\}$ is stochastically bounded.
It follows from Theorem \ref{theo:abandoncouple} and
\[
Q^n(t)=Q^n(0)+A^n(t)-K^n(t)-G^n(t)
\]
that $\bar K^n\dto \bar e$ as $n\to\infty$. By Lemma
\ref{lem:equivalent}, we have
\begin{eqnarray}
\sup_{0\leq t\leq T}\Big|\tilde M^n_1(t)-\tilde
N^n_1(t)\Big|\Rightarrow 0. \label{2014-march-24-15}
\end{eqnarray}
Therefore, it follows from
\eqref{2013-dec-21}--\eqref{2014-march-24-15} that
\begin{eqnarray}
(\tilde X^n(0), \tilde E^n, \tilde M^n_1,\tilde M^n_0)\Rightarrow
(\xi, \tilde E, \tilde N_1, \tilde N_0). \label{2013-dec-21-1}
\end{eqnarray}

By Proposition 2.1 in \cite{Reed2009},  similar to Proposition~4.4
of \cite{ReedShaki2014}, and in view of \eqref{service}, we have
\begin{eqnarray*}
  \label{eq:S-hat-n}
  \tilde D^n(t) = -\left( \tilde M^n_0(t)+ \tilde M^n_1(t)\right)
  -\int_0^t \left( \tilde M^n_0(t-s)+ \tilde M^n_1(t-s)\right) dM(s).
\end{eqnarray*}
Thus, the proposition follows directly from \eqref{2013-dec-21-1}
and the continuous mapping theorem.
\end{proof}

\subsection{$G/M/n^\alpha+GI$ in the NDS Regime}
\label{sec:NDS}

In this subsection, we consider the sequence of $G/M/n^\alpha+G$ queues with $\alpha\in (0, 1)$ in the NDS regime considered in \cite{Atar2012NDS} and \cite{GurvichAtar2014}.
Again, the speed for each server to process customer service requirements is assumed to be one.
The following assumption on the customer service times is needed.
\begin{assumption}\label{assump:Atar}
  For the $n$th system $G/M/n^\alpha+G$, the customers' remaining service requirements and service requirements $\{v^n_i, i\geq -X^n(0)+1\}$ are independent and exponentially distributed with parameter $\mu^n=n^{1-\alpha}\mu$.
\end{assumption}

By the memoryless property of customer service times, as usual, the evolution of the process $X^n$ can be characterized by the system dynamics equation
\begin{equation*}
  X^n(t)=X^n(0)+E^n(t)-S_p\Big(\mu^n\int_0^t( X^n(s)\wedge n^\alpha) ds\Big)-G^n(t),
\end{equation*}
where $S_p(\cdot)$ is a Poisson process with rate one.
Since
\[
n\mu t-\mu^n\int_0^t (X^n(s)\wedge n^{\alpha})ds=\mu^n\int_0^t (X^n(s)-n^{\alpha})^-ds,
\]
we have
\begin{equation}\label{equa:SystemDynamic}
  \begin{split}
    X^n(t)-n^\alpha &= X^n(0)-n^\alpha+E^n(t)-\lambda^n t\\
    &\quad-\Big[S_p\Big(\mu^n\int_0^t(X^n(s)\wedge n^\alpha)ds\Big)
      -\mu^n\int_0^t(X^n(s)\wedge n^\alpha)ds\Big]
    -G^n(t)\\
    &\quad+(\lambda^n-n\mu)t+\mu^n\int_0^t(X^n(s)-n^\alpha)^-ds.
  \end{split}
\end{equation}
Applying the diffusion scaling for $X^n,E^n$ and $G^n$ and the definition of $\beta^n$ in \eqref{eq:HT-QED}, we obtain
\begin{equation}
  \label{eqn:diffdyn-1-NDS}
 (\tilde X^n(t))^+  =
    \tilde Y^n(t)+ (\tilde X^n(t))^- -\tilde{G}^n(t) +\mu^n\int_0^t(\tilde X^n(s))^- ds,
 \end{equation}
where
\begin{align}
  \ds Y(t) &=\ds X(0) +\ds E(t) - \ds S_p(t)  +\beta^n\mu t, \label{2012-Aug-6-3}\\
 \ds S_p(t) &= \frac{1}{\sqrt n}
  \left[
    S_p\Big(\mu^n\int_0^t(X^n(s)\wedge n^\alpha)ds\Big)
    -\mu^n\int_0^t(X^n(s)\wedge n^\alpha)ds
  \right].\label{eq:service-diffusion}
\end{align}
We can see that $(\tilde X^n(t))^+$, by \eqref{eqn:diffdyn-1-NDS}, is related to the solution of the Skorohod equation. This observation is useful in establishing the stochastic boundedness of queue length processes, see the proof of Proposition \ref{proposition:serviceNDS}.
First we prove the following result.
\begin{proposition}\label{lem:negative-NDS}
  Assume that conditions~\eqref{eq:arrival-diffusion} and \eqref{eq:HT-QED}, and Assumption {\rm\ref{assump:Atar}}  hold.
  If condition \eqref{eq:X-0-diffusion} holds with $\mathbb{P}(\xi\geq 0)=1$, then $(\ds X)^- \Rightarrow 0$ as $n\to\infty$.
\end{proposition}
\begin{proof}
  The proof is similar to the one in \cite{Atar2012NDS}. For any fixed $\varepsilon>0$, we will prove that
  \begin{equation}
  \prob\Big(\sup_{0\leq t\leq T}(\ds X(t))^-\geq \varepsilon\Big) \to 0 \ \mbox{as $n\to \infty$}.
\end{equation}
Define
\begin{equation*}
\Omega^n_0(\varepsilon)=\{(\ds X(0))^-\le \frac{\varepsilon}{4}\}, \Omega^n(\varepsilon, T)=\{\sup_{0\leq t\leq T}(\ds X(t))^-\geq \varepsilon\},
t_1^n=\inf\{t\geq 0: (\ds X(t))^-\geq \varepsilon\}.
\end{equation*}
Because of $\xi \geq 0$ with probability one, by \eqref{eq:X-0-diffusion}, it is sufficient to prove that
the probabilities of the event $\Omega^n(\varepsilon, T)\cap \Omega^n_0(\varepsilon)$ vanishes as $n$ converges to infinity.
On the set $\Omega^n(\varepsilon, T)\cap \Omega^n_0(\varepsilon)$, define
$t^n_{2}=\sup\{0\le t\leq \eta^n: (\ds X(t))^-\leq {\varepsilon}/{3}\}\vee 0$.
By the definitions of $t^n_1$ and $t^n_{2}$, we clearly have that
\begin{equation*}
  (\ds X(t^n_1))^-\ge \varepsilon \ \ \textrm{and}\ \ (\ds X(t^n_{2}-))^-\le \frac{\varepsilon}{3}.
\end{equation*}
Note that $(\ds X(t))^-\geq \frac{\varepsilon}{3}$ for all $t\in[t^n_{2}, t^n_1]$. As a result, $\ds X(t)=-(\ds X(t))^-$ and there is no abandonment during this interval. From equation \eqref{eqn:diffdyn-1-NDS},
we have that on $\Omega^n(\varepsilon, T)\cap \Omega^n_0(\varepsilon)$,
 \begin{equation*}
  \begin{split}
  \tilde Y^n(t^n_1)-\tilde Y^n(t^n_2-) = & (\tilde X^n(t^n_2-))^- - (\tilde X^n(t^n_1))^--\mu^n\int_{t^n_2}^{t^n_1}(\tilde X^n(s))^- ds\\
  \leq &-\frac{2\varepsilon}{3}-\frac{\varepsilon\mu^n(t^n_1-t^n_2)}{3}.
  \end{split}
\end{equation*}
For fixed $\delta$, depending on whether $t^n_1-t^n_2> \delta$ or $t^n_1-t^n_2\leq \delta$, we get
\begin{eqnarray}
  &&\lim_{n\to\infty}\mathbb{P}\left(\Omega^n(\varepsilon, T)\cap \Omega^n_0(\varepsilon)\right)\nonumber\\
  && \ \ \ \leq \lim_{n\to\infty}\mathbb{P}\Big(\tilde Y^n(t^n_1)-\tilde Y^n(t^n_2-)\leq -\frac{2\varepsilon}{3}-\frac{\varepsilon\mu^n(t^n_1-t^n_2)}{3}\Big)\nonumber\\
  && \ \ \ \leq \lim_{n\to\infty}\mathbb{P}\Big(\sup_{\substack{0\leq s, t\leq T\\ |s-t|\leq \delta}} \Big|\tilde Y^n(t)-\tilde Y^n(s)\Big|\geq \frac{2\varepsilon}{3}\Big)+\lim_{n\to\infty}\mathbb{P}\Big(\sup_{0\leq t\leq T}\Big|\tilde Y^n(t)\Big|\geq \frac{\varepsilon\mu^n\delta}{6}\Big).
  \label{Atar-2015-1}
\end{eqnarray}
Noting that for $u, v\in [0, \infty)$ with $u<v$,
\begin{align*}
  0 &\leq \mu^n\int_0^v(X^n(s)\wedge n^\alpha)ds -\mu^n\int_0^u(X^n(s)\wedge n^\alpha)ds\\
    &=\mu^n\int_u^v(X^n(s)\wedge n^\alpha)ds \\
    &\leq n(v-u),
\end{align*}
and $\mu^n\int_0^t(X^n(s)\wedge n^\alpha)ds \leq nt$,
we know $\{\tilde S^n_p, n\in \Z_+\}$ with $\tilde S^n_p=\{ \ds S_p(t): t\geq 0\}$ given by \eqref{eq:service-diffusion}  is $C$-tight. Hence, by \eqref{eq:arrival-diffusion} and \eqref{eq:HT-QED}, we have $\{\tilde Y^n, n\in \Z_+\}$ is $C$-tight. Letting $n\to\infty$ and then $\delta\to 0$, the term in \eqref{Atar-2015-1} then converges to 0. This completes the proof.
\end{proof}

To get the diffusion approximation for the queue length processes, we need the following proposition.
\begin{proposition}
  \label{proposition:serviceNDS}
  Under the conditions required by Proposition {\rm \ref{lem:negative-NDS}}, $\tilde{S}_p^n\Rightarrow \sqrt \mu \tilde{S}_p$ as $n\to\infty$, where $\tilde{S}_p=\{\tilde S_p(t): t\geq 0\}$is a standard Brownian motion which is independent of the limit of the arrival processes $(\tilde E$ given by \eqref{eq:arrival-diffusion}$)$ as well as of the initial states $(\xi$ given by \eqref{eq:X-0-diffusion}$)$.
\end{proposition}
\begin{proof}
Consider the solution $(\tilde Z^n(\cdot), \tilde Z_r^n(\cdot))$ to the following Skorohod equation: with probability one,
\begin{eqnarray*}
&&\tilde Z^n(t)=\tilde Y^n(t)+ (\tilde X^n(t))^-  +\tilde Z_r^n(t), \ t\geq 0,\\
&& \tilde Z^n(t)\geq 0,  \ t\geq 0;\\
&&\tilde Z^n_r(\cdot) \ \ \mbox{is nondecreasing};\\
&&\int^\infty_0 \id{\tilde Z^n(t)>0} d\tilde Z_r^n(t)=0.
 \end{eqnarray*}
By $C$-tightness of  $\{\tilde Y^n, n\in \Z_+\}$ given by the proof of Proposition \ref{lem:negative-NDS}, and the Lipschitz continuity of the Skorohod mapping, it follows from Proposition \ref{lem:negative-NDS} that $\{\tilde Z^n, n\in \Z_+\}$ is stochastically bounded.
By $ \tilde Y^n(t)+ (\tilde X^n(t))^- \geq \tilde Y^n(t)+ (\tilde X^n(t))^- -\tilde{G}^n(t)$ with probability one, and \eqref{eqn:diffdyn-1-NDS}, with the help of Lemma 4.1 in \cite{KLRS2007}, we know that
$(\ds X(t))^+$ $(=\ds Q(t))$ can be bounded by $\tilde Z^n(t)$.
Therefore, $\{\ds Q, n\in \Z_+\}$ is also stochastically bounded. Thus, by Theorem \ref{theo:abandoncouple}, we have that as $n\rightarrow \infty$
\begin{eqnarray}
\bar{G}^n\Rightarrow 0 \ \ \mbox{with $\bar G^n=\{\frac{1}{n}G^n(t): t\geq 0\}$}. \label{Atar-2015-2}
\end{eqnarray}
On the other hand, by \eqref{eqn:diffdyn-1-NDS},
\begin{equation}
  \label{Atar-2015-3}
\frac{\mu^n}{\sqrt n}\int_0^t(\tilde X^n(s))^- ds
 =\frac{1}{\sqrt n} \Big((\tilde X^n(t))^+
   - \tilde Y^n(t)-(\tilde X^n(t))^- +\tilde{G}^n(t)\Big).
 \end{equation}
Combining (\ref{Atar-2015-2})-\eqref{Atar-2015-3} yields that
\begin{equation*}
  \frac{1}{n}\int_0^{\cdot}\mu^n(X^n(s)-n^\alpha)^-ds\Rightarrow 0
  \quad \hbox{as}\ n\rightarrow \infty,
\end{equation*}
which consequently implies that
\begin{equation}\label{equation:idlecumulative}
  \frac{1}{n}\int_0^{\cdot}\mu^n(X^n(s)\wedge n^\alpha)ds\Rightarrow \bar e(\cdot),
  \quad \hbox{as}\ n\rightarrow \infty,
\end{equation}
where $\bar e(t)=\mu t$.
The proposition  directly follows from \eqref{eq:service-diffusion} and  the random-time-change theorem (Corollary~1 of \cite{Whitt1980}).
\end{proof}

Now we are ready to state the diffusion approximation for the queue length processes.
\begin{theorem}
  \label{theorem:diffusion}
  Assume that conditions~\eqref{eq:arrival-diffusion}--\eqref{eq:Lipschitz} and
  \eqref{eq:X-0-diffusion}--\eqref{eq:HT-QED} hold. If Assumption {\rm \ref{assump:Atar}}
  holds and $\xi \geq 0$ with probability one, then $\tilde{X}^n\Rightarrow \tilde{X}$ as $n\to\infty$.
  Here $\tilde{X}=\{\tilde X(t): t\geq 0\}$ is
  given by $\tilde{X}=\Phi_g(\tilde Y)$ $($recall that $\Phi_g(\cdot)$ is defined in Lemma~{\rm\ref{lem:mapping-abd-single})} with
  \begin{align*}
  g(t)&=-\mu f\Big(\frac{1}{\mu}t \Big),  \ t\geq 0;\\
  \tilde Y&=\{\tilde Y(t): t\geq 0\}, \
  \tilde Y(t)= \xi +\tilde{E}(t) - \sqrt{\mu}\tilde{S}_p(t) + \beta\mu t,
  \end{align*}
    where $\tilde{S}_p$ given by Proposition {\rm \ref{proposition:serviceNDS}} is a standard Brownian motion independent of $\xi$ and $\tilde{E}$.
    Moreover, $\ds Q\dto \tilde X$ as $n\to\infty$.
\end{theorem}

\begin{proof}[Proof of Theorem {\rm\ref{theorem:diffusion}}]
First from condition \eqref{eq:X-0-diffusion} on the initial states, the condition \eqref{eq:arrival-diffusion} on the arrival process, \eqref{eq:HT-QED} on the traffic condition, Proposition \ref{lem:negative-NDS} on $(\ds X(t))^-$, and Proposition~\ref{proposition:serviceNDS} for $\{\tilde S^n_p, n\in \Z_+\}$, we have that as $n\to\infty$,
\begin{equation*}
  \ds Y+(\ds X)^-\dto \xi + \tilde E - \sqrt \mu \tilde S
  + \beta \bar e.
\end{equation*}
Note that by \eqref{eqn:diffdyn-1-NDS},
\begin{equation*}
  \label{NSD-proof}
  (\tilde X^n(t))^+ =\tilde Y^n(t)+(\ds X(t))^-  -\tilde{G}^n_c(t)-\mu \int^t_0 f(\frac{1}{\mu} (\tilde X^n(s))^+)ds +\mu^n\int_0^t(\tilde X^n(s))^- ds.
\end{equation*}
Recall that the stochastic boundedness of the queue length processes is proved in the proof of Proposition~\ref{proposition:serviceNDS}.
With $\ds Y+(\ds X)^-$ playing the role of $\ds Y$ in Theorem~\ref{thm:unified-reflection}, it follows from Lemma~\ref{lem:mapping-abd-single} and Theorem~\ref{thm:unified-reflection} that $\tilde Q^n=(\tilde X^n)^+\Rightarrow \tilde{X}$ as $n\to\infty$.
It then follows from Proposition~\ref{lem:negative-NDS} that $\tilde{X}^n\Rightarrow \tilde{X}$ as $n\to\infty$.
Hence, the proof of the theorem is completed.
\end{proof}

\begin{remark}\label{uniqueness}
Note that $\tilde{X}$ in Theorem {\rm \ref{theorem:diffusion}} has the similar structure as the one in Theorem {\rm \ref{theorem:diffusion-single}}.
 \end{remark}

\section{Proofs of   Theorems
\ref{theo:abandoncouple}-\ref{thm:unified-reflection}, and Corollary
\ref{corollary:abandon-limit}} \label{sec:outline-proof}
 In this
section, we give the proofs of the theorems and corollary given in
Section \ref{sec:fram-cust-aband}. First we look at the first
theorem, Theorem \ref{theo:abandoncouple}. The proof of
Theorem~\ref{theo:abandoncouple} is based on three properties of
such queueing systems, namely,
Propositions~\ref{prop:waiting-time-stoch-bound}--\ref{prop:little-law},
which  are of independent  interest themselves. We will first state
these properties and then apply them to prove
Theorem~\ref{theo:abandoncouple}. The proofs of the three
propositions are given in Appendix \ref{Appendix-Prop}. 

In order to describe these three propositions, following
\cite{DaiHe2010}, we introduce two notions. The first one is the
\emph{offered waiting time} $\omega^n_i$, which denotes the time
that the $i$th arriving customer in the $n$th system after time $0$
has to wait before receiving service for each $i\ge 1$. When
$Q^n(0)>0$, we index the initial customer in the queue by $0, -1,
\ldots, -Q^n(0)+1$, with customer~$-Q^n(0)+1$ being the first one in
the queue. Each $\omega^n_i$ denotes the remaining waiting time of
the  $i$th customer for $i=-Q^n(0)+1, \cdots, 0$. The second notion
is the \emph{virtual waiting time} $\omega^n(t)$, which is the
amount of time a hypothetical customer with infinite patience would
have to wait before receiving service upon arriving at time $t$ in
the $n$th system. We introduce the diffusion-scaled virtual
waiting-time process $\tilde{\omega}^n= \{\tilde{\omega}^n(t): t\geq
0\}$ as
\begin{equation*}
  \tilde{\omega}^n(t)=\sqrt{n}\omega^n(t).
\end{equation*}
The first property of interest is the stochastic boundedness of the
scaled virtual waiting time and abandonment probability.
\begin{proposition}\label{prop:waiting-time-stoch-bound}
  Under assumptions \eqref{eq:lambda-limit}--\eqref{eq:pat-distr-scale} and \eqref{eq:stoc-bdd}, the sequences of
  the scaled virtual waiting times $\{\tilde \omega^n, n\in\Z_+\}$ and the scaled abandonment probabilities
$\{\tilde F_\omega^n, n\in\Z_+\}$ given by
\[
\tilde F_\omega^n=\Big\{\sup\limits_{0\leq i\leq
E^n(t)}\sqrt{n}F^n(\omega^n_i):t\geq 0\Big\}
\]
 are stochastically
bounded for any given $T>0$.
\end{proposition}

The second proposition reveals an asymptotic relationship between the abandonment process and the offered waiting time.
\begin{proposition}\label{prop:abandon-patience-relation}
  Under assumptions \eqref{eq:lambda-limit}--\eqref{eq:pat-distr-scale} and \eqref{eq:stoc-bdd}, for each $T>0$,
  \begin{equation*}
    \sup\limits_{0\leq t\leq T}    \bigg|
    \tilde{G}^n(t)-\frac{1}{\sqrt{n}}\sum_{j=1}^{E^n(t)}F^n(\omega^n_j)
    \bigg|\Rightarrow 0,
    \quad \textrm{as } n\to\infty.
  \end{equation*}
 \end{proposition}

Note that neither of the above two propositions needs the modulus of
continuity to asymptotically vanish  as in \eqref{eq:modular-cont}.
The  next proposition establishes the relationship between the
virtual waiting time and the queue length. For this, condition
\eqref{eq:modular-cont} is required.
\begin{proposition}\label{prop:little-law}
  Under assumptions \eqref{eq:lambda-limit}--\eqref{eq:pat-distr-scale} and \eqref{eq:stoc-bdd}--\eqref{eq:modular-cont}, for each $T>0$,
  \begin{equation*}
    \sup_{0\leq t\leq T}
    \left|\mu \tilde{\omega}^n(t)-\tilde{Q}^n(t)\right|
    \Rightarrow 0,
    \quad \textrm{as } n\to\infty.
  \end{equation*}
\end{proposition}

\begin{remark}\label{workload-tight}
By the above proposition and the triangle inequality, for any $s,t\in [0, T]$,
\begin{equation*}
  \Big|\tilde{\omega}^n(t)-\tilde{\omega}^n(s)\Big|
  \leq  2\sup_{0\leq t\leq T}\Big|\tilde{\omega}^n(t)-\frac{1}{\mu}\tilde{Q}^n(t)\Big|
  + \frac{1}{\mu}\Big|\tilde{Q}^n(t)-\tilde{Q}^n(s)\Big|.
\end{equation*}
Thus, the $C$-tightness of $\{\tilde Q^n, n\in\Z_+\}$ implies the $C$-tightness of $\{\tilde w^n, n\in\Z_+\}$.
\end{remark}

\begin{remark}
  The same result has been proved by {\rm \cite{TalrejaWhitt2009}} under different assumptions.
  Theorem~$3.1$ in {\rm \cite{TalrejaWhitt2009}} requires the convergence of several scaled processes including those describing arrival, service completion, abandonment and total number of customers in the system, while our result only needs the convergence of the arrival processes and $C$-tightness of the queue length processes.
\end{remark}

\begin{proof}[Proof of Theorem \ref{theo:abandoncouple}] First
consider $C$-tightness of $\{\ds G, n\in \Z_+\}$. According to
Proposition \ref{prop:abandon-patience-relation}, it is enough to
show the $C$-tightness for $\{\frac{1}{\sqrt{n}}\sum_{i=1}^{E^n(t)}
F^{n}(\omega^{n}_i), n\in \Z_+\}$. Define the fluid-scaled arrival
process $\fs E=\{\fs E(t): t\geq 0\}$ by
\begin{equation*}
  \fs E(t)=\frac{E^n(t)}{n}.
\end{equation*}
Condition \eqref{eq:arrival-diffusion} implies that as $n\to\infty$,
\begin{equation}\label{eq:arrival-fluid}
  \fs E\Rightarrow \bar e,
\end{equation}
where $\bar e(t)=\mu t$.
Note that for any $0\le s\le t\le T$,
\begin{equation*}
  \frac{1}{\sqrt{n}}\sum_{i=E^n(s)+1}^{E^n(t)} F^{n}(\omega^{n}_i)
  \leq
  \left[\bar{E}^n(s)-\bar{E}^n(t)\right]
  \cdot \sup_{0\leq i\leq E^n(T)}\sqrt{n}F^{n}(\omega^{n}_i).
\end{equation*}
So the $C$-tightness follows from the $C$-tightness of $\bar E^n$
(due to \eqref{eq:arrival-fluid}) and the stochastic boundedness of
$\sup_{0\leq i\leq E^n(T)}\sqrt{n}F^{n}(\omega^{n}_i)$ (due to
Proposition \ref{prop:waiting-time-stoch-bound}).

Next we look at (\ref{equation:abandoncouple}). According to
Proposition \ref{prop:abandon-patience-relation}, it is enough to
prove that  as $n\rightarrow \infty$,
\begin{equation}
  \sup\limits_{0\leq t\leq T} \Big|    \frac{1}{\sqrt{n}}
  \sum\limits_{j=1}^{E^n(t)}F^n(\omega^{n}_j)-\mu\int_0^tf(\frac{1}{\mu}\tilde{Q}^n(s))
  ds
    \Big|\Rightarrow 0.
\end{equation}
After adding and subtracting a new term, we have
\begin{equation*}
\begin{split}
&\frac{1}{\sqrt{n}}\sum\limits_{j=1}^{E^n(t)}F^n(\omega^{n}_j)-
\mu \int_0^tf (\frac{1}{\mu }\tilde{Q}^n(s)) ds\\
& \ \ \ =\frac{1}{\sqrt{n}}\sum\limits_{j=1}^{E^n (t)}F^n (\omega^{n}_j)-
\mu \int_0^tf (\tilde{\omega}^{n}(s)) ds +\mu \int_0^tf (\tilde{\omega}^{n}(s)) ds-\mu \int_0^tf (\frac{1}{\mu }\tilde{Q}^n (s)) ds.
\end{split}
\end{equation*}
Thus, it is enough to prove that when $n\rightarrow \infty$,
\begin{align}
  \label{equa}
  \sup\limits_{0\leq t\leq T}
  \Big|\int_0^tf (\tilde{\omega}^{n}(s)) ds
    -\int_0^tf (\frac{1}{\mu }\tilde{Q}^n (s)) ds
  \Big|&\Rightarrow 0, \\
  \label{relation1}
  \sup\limits_{0\leq t\leq T}
  \Big|\frac{1}{\sqrt{n}}\sum\limits_{j=1}^{E^n (t)}F^n (\omega^{n}_j)
    -\mu\int_0^tf (\tilde{\omega}^{n}(s)) ds
  \Big|&\Rightarrow 0.
\end{align}

We first  prove \eqref{equa}. By assumption~\eqref{eq:stoc-bdd} and
Proposition~\ref{prop:waiting-time-stoch-bound}, for any
$\varepsilon>0$, there exists $\Gamma$ large enough such that for all
large enough $n$,
\begin{equation}
  \mathbb{P}\Big\{\sup_{0\leq t\leq T}\frac{1}{\mu}\tilde{Q}^n(t)\geq \Gamma\Big\}
  + \mathbb{P}\Big\{\sup_{0\leq t\leq T}\tilde{\omega}^n(t)\geq \Gamma\Big\}
  \leq \frac{\varepsilon}{2}.\label{eq:tech-prob-bound-8-6}
\end{equation}
For any $\delta>0$, by the local Lipschitz continuity of
 $f(\cdot)$ given by (\ref{eq:Lipschitz}), we have
\begin{equation*}
\begin{split}
     &\quad\mathbb{P}\Big\{\sup\limits_{0\leq t\leq T}
     \Big|\int_0^tf (\tilde{\omega}^{n}(s)) ds
       -\int_0^tf (\frac{1}{\mu }\tilde{Q}^n (s)) ds
     \Big|\geq \delta\Big\}\\
     & \ \ \ \leq \mathbb{P}\Big\{\sup_{0\leq t\leq T}\frac{1}{\mu}\tilde{Q}^n(t)\geq \Gamma\Big\}
     + \mathbb{P}\Big\{\sup_{0\leq t\leq T}\tilde{\omega}^n(t)\geq \Gamma\Big\}
     +\mathbb{P}\Big\{\sup\limits_{0\leq t\leq T}
     \Big|\tilde{\omega}^{n}(t)-\frac{1}{\mu}\tilde{Q}^n (t)\Big|
     \geq \frac{\delta}{T\Lambda_{\Gamma}}\Big\}.
\end{split}
\end{equation*}
By Proposition~\ref{prop:little-law}, the third term on the
right-hand side in the above can be less than ${\varepsilon}/{2}$
for all large enough $n$. Thus \eqref{equa} is proved by
\eqref{eq:tech-prob-bound-8-6}.

Next, we prove \eqref{relation1}. According to Lemma 3.2  of \cite{DaiHe2010} and the monotonicity of the distribution function $F^n(\cdot)$,
\begin{equation*}
  \int_0^t\sqrt{n}F^n (\frac{1}{\sqrt{n}}\tilde{\omega}^{n}(s-)) d\bar{E}^n (s)
  \leq
  \frac{1}{\sqrt{n}}\sum\limits_{j=1}^{E^n (t)}F^n (\omega^{n}_j)
  \leq
  \int_0^t\sqrt{n}F^n (\frac{1}{\sqrt{n}}\tilde{\omega}^{n}(s)) d\bar{E}^n (s).
\end{equation*}
As a result, it is sufficient to prove the following convergence as $n\to\infty$,
\addtocounter{equation}{1}
\begin{align}
  \label{relation2}\tag{\theequation a}
 \sup\limits_{0\leq t\leq T}\Big|
  \int_0^{ t}\sqrt{n}F^n (\frac{1}{\sqrt{n}}\tilde{\omega}^{n}(s)) d\bar{E}^n (s)
  -\mu \int_0^{ t}f (\tilde{\omega}^{n}(s)) ds
  \Big|\Rightarrow 0, \\
  \label{relation}\tag{\theequation b}
   \sup\limits_{0\leq t\leq T}\Big|
  \int_0^{ t}\sqrt{n}F^n (\frac{1}{\sqrt{n}}\tilde{\omega}^{n}(s-))d\bar{E}^n (s)
  -\mu \int_0^{ t}f (\tilde{\omega}^{n}(s)) ds
  \Big| \Rightarrow 0.
\end{align}

We only prove \eqref{relation2} since \eqref{relation} can be proved
similarly. The idea is similar to the one  proposed by
\cite{WardGlynn2005}. By Remark \ref{workload-tight} and
(\ref{eq:arrival-fluid}), we have that
$\{(\tilde{\omega}^{n},\bar{E}^n),n\in\Z_+\}$ is $C$-tight. So for
every convergent subsequence indexed by $n_k$,
\begin{equation*}
  (\tilde{\omega}^{n_k},\bar{E}^{n_k} )\Rightarrow (\tilde{\omega} , \bar e)
  \quad\textrm{as }n_k\to\infty,
\end{equation*}
for some process $\tilde{\omega}\in \C(\R_+, \R)$.
By the Skorohod representation theorem, there  exists another
probability space $(\breve{\Omega}, \breve{\mathcal {F}},
\breve{\mathbb{P}})$, as well as a sequence of processes
$(\breve{\omega}^{n_k},\breve{E}^{n_k} )$ and $(\breve{\omega} ,\bar
e)$ defined on it, such that
\begin{align*}
(\breve{\omega}^{n_k}, \breve{E}^{n_k} )&\stackrel{d}{=}(\tilde{\omega}^{n_k},\bar{E}^{n_k} ),\\
(\breve{\omega} ,\bar  e)&\stackrel{d}{=}(\tilde{\omega} ,\bar e),
\end{align*}
and with probability one, $\breve{\omega}^{n_k}$ converges to
$\breve{\omega}$  in $\D(\R_+, \R)$ and $\breve{E}^{n_k}$ converges
to $\bar e$ in $\D(\R_+, \R)$.
We have that for any $T\geq 0$,
\begin{align}
  &\quad \sup\limits_{0\leq t\leq T}\Big|
  \int_0^{t}\sqrt{n_k}F^{n_k} (\frac{1}{\sqrt{n_k}}\tilde{\omega}^{n_k}(s)) d\bar{E}^{n_k} (s)
  -\mu \int_0^{t}f (\tilde{\omega}^{n_k}(s)) ds
  \Big|\nonumber\\
  & \ \ \stackrel{d}{=}
    \quad \sup\limits_{0\leq t\leq T}\Big|
    \int_0^t
        \Big(
          \sqrt{n_k}F^{n_k} (\frac{1}{\sqrt{n_k}}\breve{\omega}^{n_k}(s))
          -f(\breve{\omega}^{n_k}(s))
        \Big)
        d\breve{E}^{n_k}(s)
  \Big|\label{equation:difference}\\
      & \ \ \ +\sup_{0\leq t \leq T} \Big|
  \int_0^t f(\breve{\omega}^{n_k}(s))
  d\breve{E}^{n_k}(s)
  - \mu  \int_0^tf (\breve{\omega} (s))ds
  \Big|+
  \mu\sup_{0\leq t \leq T}\Big|
  \int_0^tf (\breve{\omega}^{n_k}(s))ds -\int_0^tf (\breve{\omega} (s))ds
  \Big|.\nonumber
\end{align}
There exist $N_1$ and $M$ such that when $n_k\geq N_1$,
\begin{equation*}
  \sup_{0\leq t\leq T}|\breve{E}^{n_k}(t)|\leq 2 \mu T, \quad \sup_{0\leq t\leq T}|\breve{\omega}^{n_k}(t)|\leq M.
\end{equation*}
As a result of Lemma~4.1 of \cite{Dai1995a} and Condition~\eqref{eq:pat-distr-scale},  $\sqrt{n_k}F^{n_k} (\frac{x}{\sqrt{n_k}})$ converge to $f(x)$ uniformly on compact sets.
Thus, for any given $\varepsilon>0$, we can find an $N_2$ such that when $n_k\geq N_2$,
\begin{align*}
  \sup_{0\leq x\leq M}
  \left|
    \sqrt{n_k}F^{n_k} (\frac{x}{\sqrt{n_k}})
    -f(x)
  \right|
  \leq \frac{\varepsilon}{2\mu T}.
\end{align*}
So we can conclude that
\begin{align*}
      \sup_{0\leq t\leq T}
      \bigg|
        \int_0^t
        \Big(
          \sqrt{n_k}F^{n_k} (\frac{1}{\sqrt{n_k}}\breve{\omega}^{n_k}(s))
          -f(\breve{\omega}^{n_k}(s))
        \Big)
        d\breve{E}^{n_k}(s)
      \bigg| \leq \varepsilon,
\end{align*}
for all $n_k\geq \max(N_1, N_2)$.
This proves that the first term in \eqref{equation:difference} converges to 0.
By the continuous mapping theorem and \eqref{eq:pat-distr-scale}, we also know that with probability one, $f(\breve{\omega}^{n_k})$ converges to $f(\breve{\omega})$ as $n_k\to\infty$ in $\D(\R_+, \R)$.
By Lemma 8.3 of \cite{DaiDai1999}, we know that with probability one, as
$n_k\to\infty$,
\begin{align*}
  \sup_{0\leq t \leq T} \Big|
  \int_0^t f(\breve{\omega}^{n_k}(s))
  d\breve{E}^{n_k}(s)
  - \mu  \int_0^tf (\breve{\omega} (s))ds
  \Big|&\to 0,\\
  \sup_{0\leq t \leq T}\Big|
  \int_0^tf (\breve{\omega}^{n_k}(s))ds -\int_0^tf (\breve{\omega} (s))ds
  \Big|& \to 0.
\end{align*}
As a result, with probability one, as $n_k\to\infty$,
\begin{equation}\label{eq:tech-skorohod-mapping-1}
  \sup_{0\leq t \leq T}\Big|
  \int_0^t\sqrt{n_k}F^{n_k} (\frac{1}{\sqrt{n_k}}\breve{\omega}^{n_k}(s))
  d\breve{E}^{n_k}(s)-
  \mu \int_0^tf (\breve{\omega}^{n_k}(s))ds
  \Big| \to 0.
\end{equation}
Since $(\breve{\omega}^{n_k}, \breve{E}^{n_k})\stackrel{d}{=}(\tilde{\omega}^{n_k},\bar{E}^{n_k} )$, we have
\begin{equation*}
  \begin{split}
    &\quad\sqrt{n_k}\int_0^t
    F^{n_k} (\frac{1}{\sqrt{n_k}}\breve{\omega}^{n_k}(s))d\breve{E}^{n_k} (s)
    -\mu \int_0^tf (\breve{\omega}^{n_k}(s))ds\\
    & \ \ \ \ \ \ \stackrel{d}{=}
    \sqrt{n_k}\int_0^tF^{n_k} (\frac{1}{\sqrt{n_k}}\breve\omega^{n_k}(s))d\bar{E}^{n_k} (s)
    -\mu \int_0^tf (\tilde{\omega}^{n_k}(s))ds.
  \end{split}
\end{equation*}
Hence, \eqref{eq:tech-skorohod-mapping-1} implies that as $k\to\infty$,
\begin{equation*}
  \sup_{0\leq t \leq T}\Big|
  \sqrt{n_k}\int_0^tF^{n_k} (\frac{1}{\sqrt{n_k}}\breve\omega^{n_k}(s))d\bar{E}^{n_k} (s)
  -\mu \int_0^tf (\tilde{\omega}^{n_k}(s))ds
  \Big|\Rightarrow 0.
\end{equation*}
Since the above convergence to zero holds for all convergent
 subsequences, \eqref{relation2} is established.
\end{proof}

\begin{proof} [Proof of  Corollary
\ref{corollary:abandon-limit}] If
$F^n(x)=1-\exp(-\int_0^xh(\sqrt{n}s)ds)$, we have
$f(x)=\int_0^xh(s)ds$. Hence,
\[
\mu\int_0^tf(\frac{1}{\mu}\tilde{Q}^n(s)) ds
=\mu\int^t_0\int^{\tilde Q^n(s)/\mu}_0 h(u)duds= \int^t_0
\int^{\tilde Q^n(s)}_0 h\Big(\frac{u}{\mu}\Big) du ds.
\]
This, by Theorem \ref{theo:abandoncouple}, implies (i)

Now we prove (ii). Note that if $F^n(x)=F(x)$ with derivative
$F'(0+)$ at $x=0$, then $f(x)=F'(0+)x$. It follows from Theorem
\ref{theo:abandoncouple} that
\begin{equation}
    \label{equation:abandoncouple-limit-3}
    \sup_{0\leq t\leq T}\Big|
      \tilde{G}^n(t)- F'(0+)\int^t_0 \tilde Q^n(s) ds
      \Big|
      \Rightarrow 0,\quad \textrm{as}\ n\rightarrow\infty.
  \end{equation}
\end{proof}

\begin{proof} [Proof of Theorem \ref{thm:unified-reflection}]
To better understand the proof, we depict its logic in
Figure~\ref{fig:road-map}.
\begin{figure}[h]
\centering \tikzstyle{block} = [draw, rectangle, text centered, text
width=5.6em,
    minimum height=4em, minimum width=5em]

\begin{tikzpicture}[auto, node distance=2cm,>=latex']
    \node [block] (s1) {\small Stochastic Boundedness of Queue};
    \node [block, right of=s1, node distance=5.1cm] (s2) {\small Tightness of Abandonment Processes};
    \node [block, right of=s2, node distance=6.6cm] (s3) {\small Tightness of Queue};
    \node [block, below of=s3, node distance=3.2cm] (s4) {\small Asymptotic Relationship \eqref{equation:abandoncouple}};
    \node [block, left of=s4, node distance=6.6cm] (s5) {\small Diffusion Approximation};

    \draw [->] (s1) -- node {(a)} (s2);
    \draw [->] (s2) -- node [name=y] {(b)} (s3);
    \draw [->] (s3) -- node [name=y] {Theorem~\ref{theo:abandoncouple}} (s4);
    \draw [->] (s4) -- node [name=y] {(c)} (s5);
\end{tikzpicture}
 \caption{ The approach to the diffusion approximation.}
 \label{fig:road-map}
\end{figure}
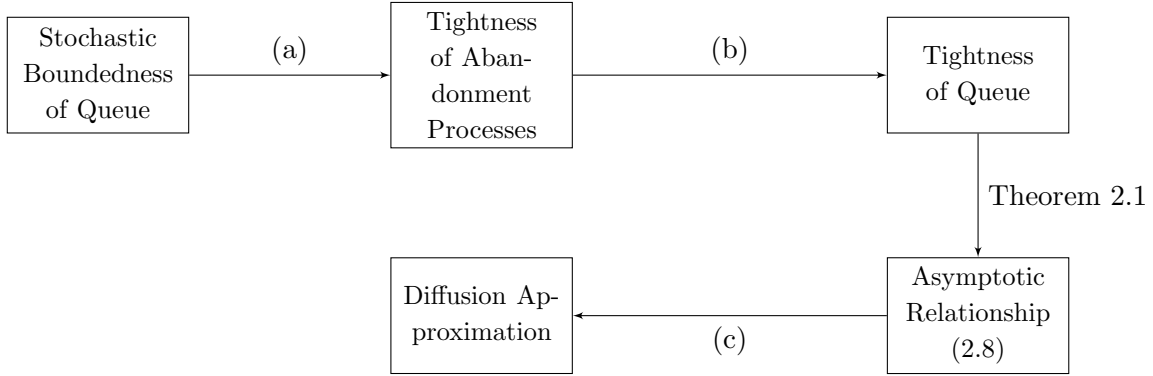

Condition (i), together with Theorem \ref{theo:abandoncouple},
implies that the abandonment process $\{\tilde G^n, n\in \Z_+\}$ is
$C$-tight (see arrow (a) in Figure~\ref{fig:road-map}). Further, the
continuity of $f(\cdot)$ implies that the sequence of processes
given by the second term on the right-hand side of \eqref{eq:G-hat}
is also $C$-tight. So according to the definition given by
\eqref{eq:G-hat}, $\{\tilde G_c^n, n\in \Z_+\}$ is $C$-tight. By
condition (ii), $\{\ds Y-\tilde G_c^n, n\in \Z_+\}$ is also
$C$-tight.
Then, for any subsequence $\{n_k, k\in \Z_+ \}$, we can find another
subsequence with indices  $\{n'_k, k\in \Z_+\}\subseteq \{n_k, k\in
\Z_+\}$, such that
\begin{equation}
  \label{eq:convergence-Y-n-k}
  \tilde{Y}^{n^\prime_k} - \tilde{G}^{n^\prime_k}_c
  \dto
  \tilde Y^\prime \ \ \textrm{as }\ n^\prime_k\to\infty,
\end{equation}
in the Skorohod $J_1$-topology for some limit $\tilde
Y^\prime\in\C(\R_+,\R)$. Recalling that if $x(\cdot)\in
\C(\R_+,\R)$, then $x_n\rightarrow x$ in the Skorohod $J_1$-topology
is equivalent to $x_n\rightarrow x$ in the uniform topology. Thus
$\Phi(\cdot)$ in condition (iii) is continuous on $\C(\R_+,\R)$ under the Skorohod
$J_1$-topology. Denote by $D_{\Phi}$ the set of discontinuous points
of $\Phi(\cdot)$ in the Skorohod $J_1$-topology. Then $\D(\R_+,\R)\setminus
\C(\R_+,\R)\supseteq D_{\Phi}$. By the condition that $\tilde
Y^\prime\in\C(\R_+,\R)$, we have $\mathbb{P}(\tilde Y^\prime\in
D_{\Phi})=0$. So the measurability of $\Phi(\cdot)$, together with
the continuous mapping theorem (see Theorem~2.7 in
\cite{Billingsley1999}, page 21) and \eqref{eq:convergence-Y-n-k},
yields
\begin{equation}
\tilde{X}^{n^\prime_k}=  \Phi(\tilde{Y}^{n^\prime_k} -
\tilde{G}^{n^\prime_k}_c)
  \dto
  \Phi(\tilde Y^\prime) \ \ \textrm{as }\ n^\prime_k\to\infty,
\end{equation}
in the Skorohod $J_1$-topology. This and $\Phi(\tilde
Y^\prime)\subseteq  \C(\R_+,\R)$ (see condition (iii)) show that
$\{\tilde X^{n_k^\prime}, k\in \Z_+\}$ is $C$-tight.    A direct
consequence is the $C$-tightness of the queue length processes
$\{\tilde{Q}^{n_k^\prime}, k\in \Z_+\}$ (see arrow (b) in
Figure~\ref{fig:road-map}). In other words, $\{\tilde
Q^{n_k^\prime}, k\in \Z_+\}$ satisfies condition
\eqref{eq:modular-cont}. From Theorem~\ref{theo:abandoncouple} and
(ii), we know $\tilde Y^\prime = \tilde Y.$
 Therefore, in view of the arbitrariness of the subsequence of $\{n_k, k\in \Z_+ \}$, we have
\begin{equation*}
  \tilde{X}^{n}\Rightarrow \Phi(\tilde{Y})
  \textrm{ \ as \ }n\rightarrow \infty
\end{equation*}
in the Skorohod $J_1$-topology (see arrow (c) in
Figure~\ref{fig:road-map}). This completes the proof.
\end{proof}

\section*{Acknowledgement}
We thank Jim Dai and Shuangchi He for suggesting the proof for Proposition~\ref{prop:waiting-time-stoch-bound}.
The research is supported by a start-up grant from NUS Business School and GRF grants (Projects No.\ 24500314, 622110 \& 622411) from the Hong Kong Research Grants Council.
We also thank the Associate Editor and two anonymous referees for their valuable suggestions which lead to Subsection~\ref{sec:NDS} and the simplification of the proof of Lemma~\ref{lemma:anadonment-CLT}.

\bibliography{pub}

\appendix

\section{Regulator Mappings}
\label{mapping}



In this section, we prove Lemmas~\ref{lem:mapping-abd-single} and \ref{lem:mapping-abd-HW} based on the following result.
\begin{lemma}\label{Lem:regulator}
  Assume that $\Upsilon: \D(\R_+,\R)\rightarrow \D(\R_+,\R)$ is measurable under the Skorohod $J_1$-topology, Lipschitz continuous under
the topology of uniform convergence over bounded intervals, and $\Upsilon(0)=0$,  and $h(\cdot)$ is  a Lipschitz continuous function on $\R$ with $h(0)=0$. Then for $y(\cdot)\in \D(\R_+,\R)$,
\begin{equation}\label{eq:mapping-ab}
  x(t)=y(t)+\int_0^t h(\Upsilon(x)(s))ds
\end{equation}
has a unique solution (denoted by $x=\Xi_{\Upsilon,h}(y)$). The mapping $\Xi_{\Upsilon,h}: \D(\R_+,\R)\rightarrow \D(\R_+,\R)$ is Lipschitz continuous under the topology of uniform convergence over bounded intervals, and measurable under the Skorohod $J_1$-topology, and  $\Xi_{\Upsilon,h}\left(\C(\R_+,\R)\right)\subseteq \C(\R_+,\R)$.
\end{lemma}

\begin{proof}
We will prove this lemma in three steps:
\begin{enumerate}
\item[(a)] the existence  and uniqueness of the solution to \eqref{eq:mapping-ab};
\item[(b)] $\Xi_{\Upsilon,h}$ is Lipschitz continuous with the topology of uniform convergence over bounded intervals;
\item[(c)] $\Xi_{\Upsilon,h}$ is measurable with respect to the Borel $\sigma$-field generated by the Skorohod $J_1$-topology.
\end{enumerate}
We focus our analysis on the bounded interval $[0,T]$ for some
$T>0$. Let $\Lambda^h$ be the Lipschitz constant of $h(\cdot)$  and $\Lambda_T^{\Upsilon}$ be the Lipschitz constant of $\Upsilon(\cdot)$ on the
interval $[0, T]$. Let $\delta=2/(3\Lambda^{\Upsilon}_T\Lambda^h)$.

\paragraph{Proof of (a):}
 We first show the existence of a solution. Define $u_0(\cdot)=0$ and $u_n(\cdot)$
iteratively by
\begin{equation*}
  u_{n+1}(t)=y(t)+\int_0^t h(\Upsilon(u_n)(s))ds
\end{equation*}
for all $n\geq 0$. Then
\begin{equation}\label{2013-Aug-1}
  \begin{split}
    u_{n+1}(t)-u_n(t)
    &=\int_0^t\big[h(\Upsilon(u_n)(s))-h(\Upsilon(u_{n-1})(s))\big]ds.
  \end{split}
\end{equation}
Now we will show that
\begin{equation}
  \label{equa:induction}
  \|u_{n+1}-u_n\|_{j\delta}\leq j^jn^j(\frac{2}{3})^n\|y\|_{(\fl{\delta^{-1}T}+1)\delta}
  \textrm{ \ for \ } j=1,2\cdots, \fl{\delta^{-1}T}+1.
\end{equation}

For $j=1$,
  \begin{equation*}
    \|u_{n+1}-u_n\|_{\delta}
    \leq  \Lambda^h\Lambda_T^{\Upsilon}\|u_n-u_{n-1}\|_{\delta}\times \delta
    \leq \frac{2}{3}\|u_n-u_{n-1}\|_{\delta}.
  \end{equation*}
  Since $h(0)=0$ and $\Upsilon(0)=0$, we have $\|u_1-u_0\|_{\delta}=\|y\|_{\delta}\leq \|y\|_{(\fl{\delta^{-1}T}+1)\delta}$. As a result,
  \begin{equation*}
    \|u_{n+1}-u_n\|_{\delta}
    \leq (\frac{2}{3})^n\|y\|_{(\fl{\delta^{-1}T}+1)\delta}
    \leq n(\frac{2}{3})^n\|y\|_{(\fl{\delta^{-1}T}+1)\delta}.
  \end{equation*}

Now assume that we have proved \eqref{equa:induction} for $j\leq k$.
Then for $j=k+1$, by (\ref{2013-Aug-1}),
  \begin{align*}
    \|u_{n+1}-u_n\|_{(k+1)\delta}
      &\leq \sum_{j=1}^k  \Lambda^h\Lambda_T^{\Upsilon}\delta\|u_n-u_{n-1}\|_{j\delta}
      + \Lambda^h\Lambda_T^{\Upsilon}\delta\|u_n-u_{n-1}\|_{(k+1)\delta}\\
      &= \sum_{j=1}^k \frac{2}{3}\|u_n-u_{n-1}\|_{j\delta}
      +\frac{2}{3}\|u_n-u_{n-1}\|_{(k+1)\delta} \\
      &\leq \sum_{j=1}^k \frac{2}{3}j^j(n-1)^j(\frac{2}{3})^{n-1}\|y\|_{(\fl{\delta^{-1}T}+1)\delta}
      +\frac{2}{3}\|u_n-u_{n-1}\|_{(k+1)\delta} \\
      &\leq k^{k+1}n^k(\frac{2}{3})^n\|y\|_{(\fl{\delta^{-1}T}+1)\delta}
      +\frac{2}{3}\|u_n-u_{n-1}\|_{(k+1)\delta}.
  \end{align*}
  Since $\|u_1-u_0\|_{(k+1)\delta}\leq \|y\|_{(\fl{\delta^{-1}T}+1)\delta}$, we have
  \begin{equation*}
    \begin{split}
      \|u_{n+1}-u_n\|_{(k+1)\delta}
      &\leq k^{k+1}(\sum_{i=0}^ni^k)(\frac{2}{3})^n\|y\|_{(\fl{\delta^{-1}T}+1)\delta}\\
      &\leq
      (k+1)^{k+1}n^{k+1}(\frac{2}{3})^n\|y\|_{(\fl{\delta^{-1}T}+1)\delta}.
    \end{split}
  \end{equation*}
Hence, we have proved \eqref{equa:induction}, which implies
\begin{equation*}
  \begin{split}
    \sum\limits_{n=1}^{\infty}\|u_{n+1}-u_n\|_T
    &\leq \sum\limits_{n=1}^{\infty}\|u_{n+1}-u_n\|_{(\fl{\delta^{-1}T}+1)\delta}\\
    &\leq \sum\limits_{n=1}^{\infty}
    (\fl{\delta^{-1}T}+1)^{\fl{\delta^{-1}T}+1}n^{\fl{\delta^{-1}T}+1}
    (\frac{2}{3})^{n}\|y\|_{(\fl{\delta^{-1}T}+1)\delta}\\
    &<\infty.
  \end{split}
\end{equation*}
Thus, $\{u_n(\cdot), n\in \Z_+\}$ is a Cauchy sequence. As $\D(\R_+,\R)$ is a Banach
Space in the uniform metric, the sequence $\{u_n(\cdot), n\in \Z_+\}$ converges to
the limit $u(\cdot)$, which is a solution to \eqref{eq:mapping-ab}.

 The uniqueness of the solution is an immediate consequence of the Lipschitz continuity of $\Xi_{\Upsilon,h}$, which we will prove next.

\paragraph{Proof of (b):}
For any $y_1(\cdot),y_2(\cdot)\in \D(\R_+,\R)$, the definition of $\delta$ and \eqref{eq:mapping-ab} also imply that
\begin{equation*}
  \|\Xi_{\Upsilon,h}(y_2)-\Xi_{\Upsilon,h}(y_1)\|_\delta
  \leq \|y_2-y_1\|_\delta+\frac{2}{3}\|\Xi_{\Upsilon,h}(y_2)-\Xi_{\Upsilon,h}(y_1)\|_\delta,
\end{equation*}
Hence, $\|\Xi_{\Upsilon,h}(y_2)-\Xi_{\Upsilon,h}(y_1)\|_\delta\leq 3\|y_2-y_1\|_\delta$.
Suppose, for $i=0,1,\ldots, k$,
\begin{equation}
  \label{eq:tech-map-induction-1}
  \|\Xi_{\Upsilon,h}(y_2)-\Xi_{\Upsilon,h}(y_1)\|_{i\delta}\leq (3i)^i\|y_2-y_1\|_{i\delta}.
\end{equation}
We now show that \eqref{eq:tech-map-induction-1} holds for $i=k+1$.
For any $t\in[0,(k+1)\delta]$, by the induction assumption and the definition of $\delta$, we have
\begin{equation*}
  \|\Xi_{\Upsilon,h}(y_2)-\Xi_{\Upsilon,h}(y_1)\|_t
  \leq \|y_2-y_1\|_t
  +\sum\limits_{i=1}^k\frac{2}{3}(3i)^i\|y_2-y_1\|_t
  +\frac{2}{3}\|\Xi_{\Upsilon,h}(y_2)-\Xi_{\Upsilon,h}(y_1)\|_t.
\end{equation*}
This implies that \eqref{eq:tech-map-induction-1} holds for $i=k+1$.
Continuing the induction until $k=\fl{\delta^{-1}T}$ yields the Lipschitz continuity property of $\Xi_{\Upsilon,h}(\cdot)$.

\paragraph{Proof of (c):}
Define
\begin{equation*}
  \Theta(y,u)(t)=y(t)+\int_0^th(\Upsilon(u)(s))ds.
\end{equation*}
First, we prove the function $\Theta(\cdot)$ is measurable with respect to the
Borel $\sigma$-field generated by the Skorohod $J_1$-topology in
$\D^2(\R_+,\R)$ and $\D(\R_+,\R)$. Define $\Pi(y,
u)(t)=y(t)+\int_0^t h(u(s))ds$. It is clear that $\Pi(\cdot)$ is
measurable (in fact, continuous) in the Skorohod $J_1$-topology. Since
$\Theta(y,u)= \Pi(y,\Upsilon(u))$ and $\Upsilon(\cdot)$ are measurable under the Skorohod $J_1$-topology, the measurability of $\Theta(\cdot)$ is
proved.

We know that $\Phi_g(y)=\lim\limits_{n\rightarrow \infty}\Theta^n(y,0)$, where $\Theta^n(\cdot)$ is iteratively defined by
\begin{equation*}
  \Theta^n(y,u)=\Theta(y,\Theta^{n-1}(y,u)),
  \quad n=1,2,\ldots,
\end{equation*}
with $\Theta^0(y,u)=u$. According to Theorem~2 on page 14 of
\cite{ChowTeicher2003}, we can prove by induction that $\Theta^n(y,0)$
is measurable for each $n$.  Since $\Xi_{\Upsilon, h}(y)$ is the
limit of $\Theta^n(y,0)$ under the topology of uniform convergence over bounded intervals,
 it is also the limit of $\Theta^n(y,0)$ under the Skorohod $J_1$-topology. By
Theorem~4.2.2 of \cite{Dudley2002}, we know that $\Xi_{\Upsilon, h}(\cdot)$ is
measurable with respect to the Borel $\sigma$-field
generated by the Skorohod $J_1$-topology.
\end{proof}

\begin{proof}[Proof of Lemma~\ref{lem:mapping-abd-single}]
Note  the special case where $g\equiv 0$ gives us the conventional Skorohod mapping.
In other words, for any $y\in \D(\R_+,\R)$ with $y(0)\geq 0$, there exists a unique $(a, b)\in \D(\R_+,\R)$ such that
\begin{equation*}
\begin{split}
    &a(t) =y(t)+b(t),\\
    &\int_0^{\infty} a(t)db(t)=0,\\
    &a(t)\geq 0, \quad \forall t\ge 0.
\end{split}
\end{equation*}
Define the mapping $(\Phi, \Psi):\D(\R_+,\R)\to \D(\R_+,\R^2)$ by $(\Phi, \Psi)(y)=(a,b)$. Then $(\Phi, \Psi)$ is
Lipschitz continuous in the topology of uniform convergence over
bounded intervals and the Skorohod $J_1$-topology.

In order to deal with the integral
equation~\eqref{2013-Nov-1}, we use Lemma~\ref{Lem:regulator} with $h(\cdot)=g(\cdot)$ and $\Upsilon(\cdot)=\Phi(\cdot)$. Then, we can obtain a mapping
$\bar{\Phi}(\cdot)$ given by $u=\bar{\Phi}(y)$ with
\begin{equation}
  u(t)=y(t)+\int_0^t g(\Phi(u)(s))ds \ \ \mbox{for $y\in \D(\R_+,\R)$}.
\end{equation}
Clearly, $(x, z)=(\Phi(\bar{\Phi}(y)), \Psi(\bar{\Phi}(y)))$ (that is, $\Phi_{g}(\cdot)=(\Phi, \Psi)\circ
\bar{\Phi}(\cdot)$) is a solution to \eqref{2013-Nov-1}. Other properties (measurability, Lipschitz continuity and $\Phi_{g}(\C(\R_+,\R))\subseteq \C(\R_+,\R)$) can easily be proved from the properties of $(\Phi, \Psi)(\cdot)$ and $\bar{\Phi}(\cdot)$.
\end{proof}


\begin{proof}[Proof of Lemma~\ref{lem:mapping-abd-HW}]
Note the special case where $g(\cdot)\equiv 0$  was proved by \cite{Reed2007} (cf.\ Proposition~7 there).
In other words, for any $y(\cdot)\in \D(\R_+,\R)$, there exists a unique $u(\cdot)\in \D(\R_+,\R)$ such that
\begin{equation*}
  u(t)=y(t)+\int_0^t \left(u(t-s)\right)^-dM(s).
\end{equation*}
Define the mapping $\Phi_M(\cdot)$ by $u(\cdot)=\Phi_M(y)(\cdot)$. Then $\Phi_M(\cdot)$ is
Lipschitz continuous under the topology of uniform convergence over
bounded intervals, measurable with respect to the Borel
$\sigma$-field generated by the Skorohod $J_1$-topology.

In order to deal with the integral
equation~\eqref{eq:mapping-general},
we use Lemma~\ref{Lem:regulator} with $h(t)=g(t^+)$ for $t\in \R$, and $\Upsilon(\cdot)=\Phi_M(\cdot)$. Then, we can obtain a mapping
$\Psi_g(\cdot)$ given by $a(\cdot)=\Psi_g(y)(\cdot)$ with
\begin{equation}
  a(t)=y(t)+\int_0^t g(\left(\Phi_M(a)(s)\right)^+)ds \ \ \mbox{for $y\in \D(\R_+,\R)$}.
\end{equation}
Clearly, $x(\cdot)=\Phi_M(\Psi_g(y))(\cdot)$ (that is, $\Phi_{M,g}(\cdot)=\Phi_M \circ
\Psi_g(\cdot)$) is a solution to \eqref{eq:mapping-general}. The other properties (measurability, Lipschitz continuity, and $\Phi_{M, g}(\C(\R_+,\R))\subseteq \C(\R_+,\R)$) can easily be proved from the properties of $\Phi_M(\cdot)$ and $\Psi_g(\cdot)$.
\end{proof}

\section{Technical Proofs}

\subsection{Proof of Lemma \ref{lem:equivalent}}

\begin{proof}[Proof of Lemma \ref{lem:equivalent}]
First introduce auxiliary processes $\tilde {\cal T}^n_0=\{ \tilde {\cal T}^n_0(t): t\geq 0\}$ and $\tilde  U^n_0=\{ \tilde  U^n_0(t): t\geq 0\}$
\begin{align*}
  \tilde {\cal T}^n_{0}(t) &= \frac{1}{\sqrt n} \sum_{i=1}^{\lfloor n\mu (T+1) \rfloor}
    \Big( \id{u_i^n+\tau_i^n>t}-(1-H_{\star}(t-\tau^n_i))\Big),\displaybreak[3]\\
  \tilde {U}^n_{0}(t) &= \frac{1}{\sqrt n} \sum_{i=1}^{\lfloor n\mu (T+1) \rfloor}
    \Big( \id{u_i^n+\frac{i}{n\mu}>t}-(1-H_{\star}(t-\frac{i}{n\mu}))\Big).
\end{align*}
By the weak convergence of the empirical processes (see Chapter 3 in \cite{ShorackWellner2009}), as $n\to\infty$,
\begin{equation}
  \label{2013-12-31-1}
  \tilde  U^n_0  \Rightarrow \tilde  U,
\end{equation}
where $\tilde  U=\{\tilde U(t): t\geq 0\}$ is a Gaussian process with continuous sample  paths, zero mean  and the same covariance function as $\tilde {\cal T}$.
By \eqref{eq:C-conv-e}, we have
\begin{eqnarray}
\sup_{0\leq t\leq T}  \Big| \tilde {\cal T}^n(t)-\tilde {\cal T}^n_{0}(t)\Big|\Rightarrow 0 \ \
\mbox{and} \ \ \sup_{0\leq t\leq T}  \Big| \tilde  U^n(t)-\tilde  U^n_{0}(t)\Big|\Rightarrow 0.
\label{2014-march-24-3}
\end{eqnarray}
In view of \eqref{2013-12-31-1} and \eqref{2014-march-24-3}, to prove the lemma it remains to be shown that
\begin{equation}
  \label{2014-march-24-4}
  \sup_{0\leq t\leq T}  \Big| \tilde {\cal T}^n_0(t)-\tilde U^n_0(t)\Big|\Rightarrow 0.
\end{equation}
According to \eqref{2013-12-31-1} and Theorem 13.1 of \cite{Billingsley1999}, the proof of \eqref{2014-march-24-4} is derived in two steps. 

\paragraph{Step 1.} Establish the convergence of all the finite dimensional distributions of $\{\tilde{\cal T}_0^n, n\in \Z_+\}$ with the same limit as $\{\tilde{\cal U}_0^n, n\in \Z_+\}$.

By \eqref{eq:C-conv-e}, for any $T>0$,
\begin{equation*}
  \max_{1\leq i \leq \lfloor n\mu T \rfloor}
  \Big|\tau^n_i-\frac{i}{n\mu}\Big|
  \Rightarrow 0.
\end{equation*}
This implies that  there is a positive sequence $\{\varepsilon^n, n\in \Z_+\}$ with $\varepsilon^n \downarrow 0$ such that
\begin{equation}
  \label{2014-feb-2}
  \sup_{0\leq t\leq T}\Big| \tilde {\cal T}^n_0(t) +\tilde {\cal T}^n_1(t,\varepsilon^n)\Big|\Rightarrow 0
  \ \ \mbox{and} \ \
  \sup_{0\leq t\leq T}\Big| \tilde U^n_0(t) +\tilde U^n_1(t,\varepsilon^n)\Big|  \Rightarrow 0,
\end{equation}
where
\begin{align*}
  \tilde {\cal T}^n_1(t,\varepsilon^n)
  &= \frac{1}{\sqrt n} \sum_{i=1}^{\lfloor n\mu (T+1) \rfloor}
     \id{\frac{i-1}{n\mu}-\varepsilon^n\leq \tau^n_i\leq \frac{i}{n\mu}+\varepsilon^n}
     \Big(\id{\tau_i^n+u_i^n\leq t}-H_{\star}(t-\tau^n_i)\Big),\\
  \tilde U^n_1(t,\varepsilon^n)
  &= \frac{1}{\sqrt n}\sum_{i=1}^{\lfloor n\mu (T+1) \rfloor}
     \id{\frac{i-1}{n\mu}-\varepsilon^n\leq \tau^n_i\leq\frac{i}{n\mu}+\varepsilon^n}
     \Big(\id{\frac{i}{n\mu}+u_i^n\leq t}-H_{\star}(t-\frac{i}{n\mu})\Big).
\end{align*}
Thus, for this step, we just need to show that for each $t\leq T$, as $n\to\infty$,
\begin{equation}
  \label{2014-march-24-1}
  \tilde {\cal T}^n_1(t,\varepsilon^n)-\tilde U^n_1(t,\varepsilon^n)\Rightarrow 0.
\end{equation}
Note that 
\begin{align}
  &\quad \prob \Big\{ |\tilde {\cal T}^n_1(t,\varepsilon^n)-\tilde U^n_1(t,\varepsilon^n)| >\delta\Big\} \nonumber\\
  &\leq \frac{1}{n\delta^2}\E \Big( \sum_{i=1}^{\lfloor n\mu (T+1) \rfloor}
  \id{\frac{i-1}{n\mu}-\varepsilon^n\leq \tau^n_i\leq \frac{i}{n\mu}+\varepsilon^n}
  \Big[\Big(\id{\tau_i^n+u_i^n\leq t}-H_{\star}(t-\tau^n_i)\Big)\displaybreak[3]\nonumber\\
  & {\hskip 200pt}
  -\Big(\id{\frac{i}{n\mu}+u_i^n\leq t}-H_{\star}(t-\frac{i}{n\mu})\Big)\Big]\Big)^2\displaybreak[3]\nonumber\\
  &\leq \frac{1}{n\delta^2}\sum_{i=1}^{\lfloor n\mu (T+1) \rfloor}\E \Big(
  \id{\frac{i-1}{n\mu}-\varepsilon^n\leq \tau^n_i\leq \frac{i}{n\mu}+\varepsilon^n}
  \Big[\Big(\id{\tau_i^n+u_i^n\leq t}-H_{\star}(t-\tau^n_i)\Big)\displaybreak[3]\nonumber\\
  & {\hskip 200pt}
  -\Big(\id{\frac{i}{n\mu}+u_i^n\leq t}-H_{\star}(t-\frac{i}{n\mu})\Big) \Big]^2\Big)\displaybreak[3]\nonumber\\
  &\leq \frac{4}{n\delta^2}\sum_{i=1}^{\lfloor n\mu (T+1) \rfloor}
  \Big(H_\star(t-\frac{i-1}{n\mu}+\varepsilon^n)-H_\star(t-\frac{i}{n\mu}-\varepsilon^n)\Big).
  \label{2014-march-24-2}
\end{align}
Consider the set of intervals given by
\begin{equation*}
  \Big\{ \Big((t-\frac{i}{n\mu}-\varepsilon^n)^+, \ \ (t-\frac{i-1}{n\mu}+\varepsilon^n)^+\Big]
  , i=1,\cdots, \lfloor n\mu (T+1) \rfloor\Big\}.
\end{equation*}
Note that the intervals $\Big(t-\frac{i}{n\mu}-\varepsilon^n, \  t-\frac{i-1}{n\mu}+\varepsilon^n\Big]$ and $\Big(t-\frac{j}{n\mu}-\varepsilon^n, \  t-\frac{j-1}{n\mu}+\varepsilon^n\Big]$ with $i<j$ are disjoint if $j \geq i+1+\lceil 2n\mu \varepsilon^n \rceil$.
Thus the set can be partitioned into $ 2+\lceil 2n\mu \varepsilon^n \rceil$ groups such that any two intervals in each group are disjoint. Therefore, 
\begin{align*}
&\frac{4}{n\delta^2}\sum_{i=1}^{\lfloor n\mu (T+1) \rfloor}
  \Big(H_\star(t-\frac{i-1}{n\mu}+\varepsilon^n)-H_\star(t-\frac{i}{n\mu}-\varepsilon^n)\Big)\\
  & \ \ \ \leq \frac{4}{n\delta^2} \times \Big(2+\lceil 2n\mu \varepsilon^n \rceil\Big) \rightarrow 0
  \ \ \mbox{as }n\rightarrow \infty.
\end{align*}
Thus, \eqref{2014-march-24-1} is proved by  \eqref{2014-march-24-2}.

\paragraph{Step 2.} Establish the tightness of $\{\tilde{\cal T}_0^n, n\in \Z_+\}$.

We first present a simple proof of the tightness, which is itself interest and requires that function $H_{\star}(\cdot)$ is locally Lipschitz continuous.
Note that for any $t_1\leq t\leq t_2$,
\begin{align}
  &\quad \E \Big[\Big(\tilde{\cal T}^n_0 (t )-\tilde{\cal T}^n_0 (t_1 )\Big)^2\times
                 \Big(\tilde{\cal T}^n_0  (t_2 )-\tilde{\cal T}^n_0 (t )
                 \Big)^2\Big]\displaybreak[3]\nonumber\\
  & \ \ \ \ \ = \frac{1}{n^2} \sum_{i,j=1}^{\fl{n\mu (T+1)}}
  \E\Big[\Big(
      \id{t_1<\tau_i^n+u_i^n\leq t}-[H_{\star}(t-\tau^n_i)-H_{\star}(t_1-\tau^n_i)]
    \Big)^2 \displaybreak[3]\nonumber\\
  &   {\hskip 105pt} \times  \Big(
      \id{t<\tau_j^n+u_j^n\leq t_2}-[H_{\star}(t_2-\tau^n_j)-H_{\star}(t-\tau^n_j)]
    \Big)^2 \Big]\displaybreak[3]\nonumber\\
   & \ \ \ \quad +\frac{2}{n^2}\sum_{i\neq j}\E\Big[ \Big(\id{t_1<\tau_i^n+u_i^n\leq t}-[H_{\star}(t-\tau^n_i)-H_{\star}(t_1-\tau^n_i)]\Big) \displaybreak[3]\nonumber\\
   &{\hskip 95pt} \times \Big(\id{t<\tau_i^n+u_i^n\leq t_2}-[H_{\star}(t_2-\tau^n_i)-H_{\star}(t-\tau^n_i)]\Big)\displaybreak[3]\nonumber\\
   &   {\hskip 95pt}  \times   \Big(\id{t_1<\tau_j^n+u_j^n\leq t}-[H_{\star}(t-\tau^n_j)-H_{\star}(t_1-\tau^n_j)]\Big) \displaybreak[3]\nonumber\\
   &  {\hskip 95pt} \times  \Big(\id{t<\tau_j^n+u_j^n\leq t_2}-[H_{\star}(t_2-\tau^n_j)-H_{\star}(t-\tau^n_j)]\Big) \Big]\displaybreak[3]\nonumber\\
  & \ \ \ \ \ \ \leq 3\mu^2(T+1)^2 \sup_{0\leq s\leq T}\Big(H_{\star}(t_2-s)-H_{\star}(t_1-s)\Big)^2.\nonumber
 \end{align}
When $H_{\star}(\cdot)$ is locally Lipschitz continuous, the right-hand side of the above inequality can be bounded by $\Lambda (t_2-t_1)^2$ for some constant $\Lambda$. So the tightness of $\{\tilde{\cal T}_0^n, n\in \Z_+\}$ follows from Theorem 13.5 of \cite{Billingsley1999}.

Now consider the case without the local Lipschitz continuity of $H_{\star}(\cdot)$. Note that by \eqref{2014-march-24-3}, the  tightness of $\{\tilde{\cal T}_0^n, n\in \Z_+\}$ and the tightness of $\{\tilde{\cal T}^n, n\in \Z_+\}$ are equivalent. So it is sufficient to prove the tightness of  $\{\tilde{\cal T}^n, n\in \Z_+\}$.
Let 
\begin{eqnarray*}
 \tilde{U}^n(t, x)  &=&\frac{1}{\sqrt{n}}\sum_{i=1}^{C^n(t)} \Big(\id{u^n_i\leq x}-H_{\star}(x)\Big), \quad t\geq 0, x\geq 0,\\
  \tilde{\cal T}^n_c(t)&=&\frac{1}{\sqrt{n}}\sum_{i=1}^{C^n(t)}\Big(\id{0<u^n_i\leq t-\tau^n_i}-\int_{0+}^{u^n_i\wedge (t-\tau^n_i)^+}\frac{d H_{\star}(s)}{1-H_{\star}(s-)}\Big).
 \end{eqnarray*}
Then we have
\begin{equation}\label{eq:tight-three-part}
\begin{split}
  \tilde {\cal T}^n(t)=&\int_0^t \frac{\tilde{U}^n(t-s, s-)}{1-H_{\star}(s-)}d H_{\star}(s-)-\frac{1}{\sqrt{n}}\sum_{i=1}^{C^n(t)}
  \Big(\id{u^n_i=0}-H_{\star}(0)\Big)-\tilde{\cal T}_c^n(t).
\end{split}
\end{equation}
Thus the tightness of $\{\tilde {\cal T}^n, n\in \Z_+\}$ follows from
the tightness of the three terms on the right-hand side of
\eqref{eq:tight-three-part}. The tightness of the second term  in
\eqref{eq:tight-three-part} follows directly from
(\ref{eq:C-conv-e}). For the first term in
\eqref{eq:tight-three-part}, we can divide it into two parts (for
any $\varepsilon>0$):
\begin{eqnarray}\label{eq:two-indicator}
  \int_0^t \frac{\tilde{U}^n(t-s, s-)}{1-H_{\star}(s-)}\id{H_{\star}(s-)>1-\varepsilon}d H_{\star}(s-)+\int_0^t \frac{\tilde{U}^n(t-s, s-)}{1-H_{\star}(s-)}\id{H_{\star}(s-)\leq 1-\varepsilon}d H_{\star}(s-).
\end{eqnarray}
In the same way \cite{KrichaginaPuhalskii1997} proved their Lemma~3.4, we obtain the tightness of the second term in (\ref{eq:two-indicator}) and
\begin{equation*}
  \lim_{\varepsilon\downarrow 0}\limsup_{n\to\infty}\mathbb{P} \Big(\sup_{t\leq T} \Big|\int_0^t \frac{\tilde{V}^n (t-s, s-)}{1-H_{\star}(s-)}\id{H_{\star}(s-)>1-\varepsilon}d H_{\star}(s-)\Big|>\delta \Big)=0.
\end{equation*}

Finally we prove the tightness of the third term  on the right-hand side of \eqref{eq:tight-three-part}.
 Let
 \[
 {\cal F}^n_t=\sigma\Big\{C^n(s): s\leq t\Big\} \vee \sigma \Big\{\id{\tau^n_i+u^n_i\leq s}: s\leq t, i=1,\cdots, C^n(t)\Big\}.
 \]
 Then for each positive integer $k$,
 \begin{equation*}
  \tilde{\cal T}^n_{c,k}(t)=\frac{1}{\sqrt{n}}\sum_{i=1}^{C^n(t)\wedge k}\Big(\id{0< u^n_i\leq t-\tau^n_i}-\int_{0+}^{u^n_i\wedge (t-\tau^n_i)^+}\frac{dH_{\star}(s)}{1-H_{\star}(s-)}\Big)
\end{equation*}
is an ${\cal F}^n_t$-square-integrable martingale with the predictable quadratic-variation process
\begin{equation*}
  \langle \tilde{\cal T}^n_{c,k}\rangle(t)=\frac{1}{n}\sum_{i=1}^{C^n(t)\wedge k}\int_{0+}^{u^n_i\wedge (t-\tau^n_i)^+}\frac{1-H_{\star}(s)}{(1-H_{\star}(s-))^2}dH_{\star}(s).
\end{equation*}
Then the tightness of the third term follows the same argument used by \cite{KrichaginaPuhalskii1997} in the proof of their Lemma 3.7.

Combining steps 1 and 2 yields the convergence of $\{\tilde{\cal T}_0^n, n\in \Z_+\}$ with the same limit as that of $\{\tilde{\cal U}_0^n, n\in \Z_+\}$. Therefore, we have \eqref{2014-march-24-4} which implies the lemma.
\end{proof}

\subsection{Proofs of  Propositions  \ref{prop:waiting-time-stoch-bound}--\ref{prop:little-law}}
\label{Appendix-Prop}

In this section, we provide the proofs for Propositions \ref{prop:waiting-time-stoch-bound}--\ref{prop:little-law}.
In order to prove the propositions, we need the following lemma. For each $\delta>0$, let
\begin{equation*}
  \tilde L^n_\delta(t)=\frac{1}{\sqrt n}\sum_{i=1}^{\lfloor nt \rfloor }
  \Big(\id{\gamma^n_i \leq \frac{\delta}{\sqrt n}}
    -F^n\Big(\frac{\delta}{\sqrt n}\Big)\Big).
\end{equation*}

\begin{lemma}\label{lemma:anadonment-CLT}
For a fixed $\delta>0$, if $F^n\left(\frac{\delta}{\sqrt n}\right)\rightarrow 0$ as $n\rightarrow\infty$ $($which is implied by \eqref{eq:pat-distr-scale}$)$, then for any
$T>0$,
\begin{equation*}
  \sup_{0\leq t \leq T}|\tilde L^n_\delta(t)| \Rightarrow 0,\quad \textrm{as }n\rightarrow \infty.
\end{equation*}
\end{lemma}
\begin{proof}
  Denote $p_n=F^n\left(\frac{\delta}{\sqrt n}\right)$ and
  $X_{ni}=\id{\gamma^n_i \leq \frac{\delta}{\sqrt n}} - p_n$. Then for any $\varepsilon>0$,
  \begin{eqnarray}
    \prob\Big\{ \sup_{0\leq t \leq T} |\tilde L^n_\delta(t)|>\varepsilon \Big\}
    =\prob\Big\{ \max_{1\leq k \leq \lfloor nT \rfloor}
      \Big|\tilde L^n_\delta\Big(\frac{k}{n} \Big)\Big| >\varepsilon\Big\}.
    \label{2012-May-17}
  \end{eqnarray}
  Note that for each $n$, $\{X_{ni}, i\in \Z_+\}$ are independent and identically distributed random variables with $\E \left( X_{ni} \right)^2=p_n(1-p_n)<\infty$. By the Kolmogorov's inequality  (see page 133 of \cite{ChowTeicher2003})
  \begin{align*}
    \prob\Big\{ \max_{1\leq k \leq \lfloor nT \rfloor}
      \Big|\tilde L^n_\delta\Big(\frac{k}{n} \Big)\Big| >\varepsilon\Big\}\leq
      &\frac{1}{\varepsilon^2}\E  |\tilde L^n_\delta\left(\lfloor T \rfloor \right)|^2\\
     =& \frac{1}{\varepsilon^2 n} \sum_{i=1}^{\fl{nT}} \E \left( X_{ni} \right)^2\nonumber\\
     \leq &\frac{1}{\varepsilon^2}T p_n(1-p_n)\to 0,
  \end{align*}
  according to the assumption that $p_n\rightarrow 0$ as $n\to\infty$.
  Thus the lemma follows from \eqref{2012-May-17}.
\end{proof}

\begin{proof}[Proof of Proposition~\ref{prop:waiting-time-stoch-bound}]
  First we look at the sequence of the scaled virtual waiting times $\{\tilde \omega^n, n\in\Z_+\}$. 
  According to FCFS, for any $s\in(0, \tilde{\omega}^n(t))$, customers who arrive
during the interval $(t, t+\frac{s}{\sqrt{n}})$ will not receive
service until $t+\frac{s}{\sqrt{n}}$, hence they either stay in the
queue or have abandoned by $t+\frac{s}{\sqrt{n}}$. So for any $s\in
(0, \ \tilde \omega^n(t))$,
\begin{equation}\label{eq:citeDaiHe}
  E^n\left(t+\frac{s}{\sqrt n}\right)-E^n(t)
  \leq Q^n\left(t+\frac{s}{\sqrt n}\right)
  +\sum_{i=E^n(t)+1}^{E^n(t+s/\sqrt n)}
   \id{\gamma^n_i \leq \frac{s}{\sqrt n}}.
\end{equation}
The above inequality \eqref{eq:citeDaiHe} implies that
\begin{align*}
 \tilde E^n\left(t+\frac{s}{\sqrt n}\right)-\tilde E^n(t)+\frac{\lambda^n s}{n}
 &\leq
 \tilde Q^n \left(t+\frac{s}{\sqrt n}\right)
 +\tilde L^n_{s}\left(\bar E^n \Big(t+\frac{s}{\sqrt n}\Big)\right)
 -\tilde L^n_{s}\left(\bar E^n \left(t\right) \right)\\
 & \quad +\sqrt n \cdot F^n\left(\frac{s}{\sqrt n}\right)
 \cdot \left(\bar E^n \Big(t+\frac{s}{\sqrt n}\Big) -\bar E^n(t)\right).
\end{align*}
Then
\begin{align}\allowdisplaybreaks
  \mathbb{P}\left\{\sup_{0\leq t\leq T}\tilde \omega^n(t)>s\right\}
  &\leq \mathbb{P}\left\{\inf_{0\leq t\leq T}
               \Big[\tilde E^n\left(t+\frac{s}{\sqrt n}\right)
                    -\tilde E^n(t)+\frac{\lambda^n s}{n}
                    -\tilde Q^n \left(t+\frac{s}{\sqrt n}\right)\right.
                    \nonumber\\
                    &\left. {\hskip 60pt}
                    - \tilde L^n_{s}\left(\bar E^n \Big(t+\frac{s}{\sqrt n}\Big) \right)
                    +\tilde L^n_{s}\left(\bar E^n \left(t\right) \right)
                    \right.\nonumber\\
                    &\left. {\hskip 60pt}
                    -\sqrt n \cdot F^n\left(\frac{s}{\sqrt n}\right)
                    \cdot \left(\bar E^n \Big(t+\frac{s}{\sqrt n}\Big) -\bar E^n(t)\right)
               \Big] \leq 0\right\}\nonumber \displaybreak[0]\\
  &\leq \mathbb{P}\Big\{\inf_{0\leq t\leq T}
          \left[\tilde E^n\left(t+\frac{s}{\sqrt n}\right)
            -\tilde E^n(t)+\frac{\lambda^n s}{n}\right]
          \leq \frac{\mu s}{2}\Big\}\nonumber\\
  &\quad+\mathbb{P}\Big\{ \sup_{0\leq t \leq T}
          \tilde Q^n \left(t+\frac{s}{\sqrt n}\right)
         \geq \frac{\mu s}{12} \Big\}\nonumber\\
  &\quad+\mathbb{P}\Big\{ \sup_{0\leq t \leq T}
          \Big|\tilde L^n_{s}\left(\bar E^n \Big(t+\frac{s}{\sqrt n}\Big) \right)
          -\tilde L^n_{s}\left(\bar E^n \left(t\right) \right)\Big|
          \geq \frac{\mu s}{12}  \Big\}\nonumber\\
  &\quad+\mathbb{P}\Big\{ \sup_{0\leq t \leq T}
          \sqrt n \cdot F^n\left(\frac{s}{\sqrt n}\right)
          \cdot \left(\bar E^n \Big(t+\frac{s}{\sqrt n}\Big) -\bar E^n(t)\right)
          \geq \frac{\mu s}{12} \Big\},\nonumber
\end{align}
where $\mu$ is given by \eqref{eq:lambda-limit}. It follows from
\eqref{eq:lambda-limit}--\eqref{eq:arrival-diffusion} that as
$n\rightarrow \infty$,
\begin{align}
  \label{2012-march-25-5}
  &\mathbb{P}\Big\{\inf_{0\leq t\leq T}
    \left[ \tilde E^n\left(t+\frac{s}{\sqrt n}\right)
          -\tilde E^n(t)+\frac{\lambda^n s}{n}\right]
       \leq \frac{\mu s}{2}\Big\} \rightarrow 0,
\end{align}
Lemma \ref{lemma:anadonment-CLT} and \eqref{eq:arrival-fluid} imply that as $n\to\infty$,
\begin{align}
  \label{2012-march-25-6}
  & \mathbb{P}\Big\{ \sup_{0\leq t \leq T}
    \Big| \tilde L^n_{s}\left(\bar E^n \left(t+\frac{s}{\sqrt n}\right) \right)
         -\tilde L^n_{s}\left(\bar E^n \left(t\right) \right)\Big|
       \geq \frac{\mu s}{12} \Big\}\rightarrow 0.
\end{align}
By \eqref{eq:pat-distr-scale} and \eqref{eq:arrival-fluid}, as $n\rightarrow \infty$,
\begin{align}
  \label{2012-march-25-7}
  \mathbb{P}\Big\{ \sup_{0\leq t \leq T}
    \sqrt n \cdot F^n\left(\frac{s}{\sqrt n}\right)
    \cdot \left(\bar E^n \Big(t+\frac{s}{\sqrt n}\Big) -\bar E^n(t)\right)
    \geq \frac{\mu s}{12}
  \Big\}\rightarrow 0.
\end{align}
Hence, the stochastic boundedness of $\{\tilde \omega^n, n\in\Z_+\}$
follows from assumption~\eqref{eq:stoc-bdd} and
\eqref{2012-march-25-5}--\eqref{2012-march-25-7}.

Next we look at $\{ \tilde F^n_\omega, n\in\Z_+\}$. It is sufficient
to show that for any given $T>0$,
\begin{equation}\label{2014-7-7}
  \lim_{\Gamma\rightarrow \infty}\limsup_{n\rightarrow \infty}
  \pr{\sup\limits_{0\leq i\leq E^n(T)}\sqrt{n}F^n(\omega^n_i)\geq \Gamma}=0,
\end{equation}
where $\omega^n_i$ is the offered waiting time.
According to Lemma 3.2  of \cite{DaiHe2010} and the monotonicity of the distribution function $F^n(\cdot)$,
\begin{equation}
  \label{eq:tech-prob-bound-8-6-2}
  \pr{\sup\limits_{0\leq i\leq E^n(T)}\sqrt{n}F^n(\omega^n_i)\geq \Gamma}
  \leq
  \pr{\sup\limits_{0\leq t\leq T}\sqrt{n}
    F^n(\frac{1}{\sqrt{n}}\tilde{\omega}^n(t))\geq \Gamma}.
\end{equation}
Thus, it is enough to prove
\begin{equation}\label{equation:limitwaiting}
  \lim\limits_{\Gamma\rightarrow \infty}\limsup\limits_{n\rightarrow \infty}\pr{\sup\limits_{0\leq t\leq T}
  \sqrt{n}F^n(\frac{1}{\sqrt{n}}\tilde{\omega}^n(t))\geq \Gamma}=0.
\end{equation}
Next, we note
\begin{equation*}
  \pr{\sup\limits_{0\leq t\leq T}\sqrt{n}F^n(\frac{1}{\sqrt{n}}\tilde{\omega}^n(t))\geq \Gamma}\leq \pr{\sqrt{n}F^n(\frac{1}{\sqrt{n}}\Gamma_1)\geq \Gamma}+\pr{\sup_{0\leq t\leq T}\tilde{\omega}^n(t)\geq \Gamma_1}.
\end{equation*}
For any given $\varepsilon>0$, by stochastic boundedness of
$\{\tilde{\omega}^n, n\in \Z_+\}$, we can choose $\Gamma_1$ such
that
 \begin{equation*}
   \limsup_{n\rightarrow \infty}\pr{\sup_{0\leq t\leq T}\tilde{\omega}^n(t)\geq \Gamma_1}\leq \frac{\varepsilon}{2}.
 \end{equation*}
Now, from \eqref{eq:pat-distr-scale}, for the $\Gamma_1$ fixed above,
\begin{equation*}
 \lim\limits_{\Gamma\rightarrow \infty} \limsup_{n\rightarrow \infty}\pr{\sqrt{n}F^n(\frac{1}{\sqrt{n}}\Gamma_1)
 \geq \Gamma}=0.
\end{equation*}
Thus, for any given $\varepsilon>0$, there is a $\Gamma_0$ such that when $\Gamma\geq \Gamma_0$,
\begin{equation}
  \limsup_{n\rightarrow \infty}\pr{\sup\limits_{0\leq t\leq T}\sqrt{n}F^n(\frac{1}{\sqrt{n}}\tilde{\omega}^n(t))\geq \Gamma}  \leq \varepsilon.
\end{equation}
This completes the proof of \eqref{equation:limitwaiting}. Thus
(\ref{2014-7-7}) is proved due to \eqref{eq:tech-prob-bound-8-6-2}.
Hence, the proof of the proposition is completed.
\end{proof}

An immediate consequence of (\ref{2014-7-7}) is that
  \begin{equation}
    \label{eq:tech-bound-F(w)}
    \E\Big[\sup_{0\leq i\leq E^n(T)}F^n(\omega^n_i)\Big]\to 0, \quad \textrm{as }n\to\infty.
  \end{equation}
This will help to prove
Lemma~\ref{lemma:centralizedindicationfunction}  below, which is an
extension of Proposition~4.2  of \cite{DaiHe2010}, where
$F^n(\cdot)=F(\cdot)$. The general approach of the proof is the same
whether $F^n(\cdot)$'s are the same or vary with $n$. That is, we need to use the martingale convergence theorem (cf.\ Lemma 4.3  of \cite{DaiHe2010} and \cite{Whitt2007}). The key
condition for applying the theorem is \eqref{eq:tech-bound-F(w)}. We,
thus,   present the result without repeating the proof.

\begin{lemma}\label{lemma:centralizedindicationfunction}
  Under assumptions \eqref{eq:lambda-limit}--\eqref{eq:pat-distr-scale} and \eqref{eq:stoc-bdd},
  \begin{equation*}
    \sup_{0\leq t\leq T}\Big|
       \frac{1}{\sqrt{n}}\sum_{i=1}^{\lfloor nt \rfloor}
       \left(\id{\gamma^n_i\leq \omega^n_i}-F^n(\omega^n_i)\right)
       \cdot g(\omega^n_i)
    \Big|\Rightarrow 0 \ \hbox{as}\ n\rightarrow\infty,
  \end{equation*}
  where $g(\cdot): \R_+\rightarrow \R_+$ is a Borel measurable function such that $0\leq g(t)\leq 1$ for all $t\in \R_+$.
\end{lemma}

Similar to \cite{DaiHeTezcan2010}, we define the process
\begin{equation*}
  \zeta^n(t)=\inf\{s\geq 0: s+\omega^n(s)\geq t\}.
\end{equation*}
It is clear that $\zeta^n\in \D(\R_+, \R)$ and is
nondecreasing for each $n\in \Z_+$.

\begin{lemma}\label{lem:timediff}
  Under assumptions \eqref{eq:lambda-limit}--\eqref{eq:pat-distr-scale} and \eqref{eq:stoc-bdd}, as $n\to\infty$
  \begin{equation*}
    \sup\limits_{0\leq t\leq T}|\zeta^n(t)-t|\Rightarrow 0.
  \end{equation*}
\end{lemma}
\begin{proof}
By the definition of $\zeta^n(t)$, for any $t\geq 0$,
\begin{equation*}
  0\le t-\zeta^n(t)\le \omega^n(\zeta^n(t)).
\end{equation*}
Hence,
\[
\sup_{0\leq t\leq T}|\zeta^n(t)-t| \leq \sup_{0\leq t\leq T}  \omega^n(\zeta^n(t)) \leq \sup_{0\leq t\leq T}  \omega^n( t).
\]
Thus, the result follows from Proposition \ref{prop:waiting-time-stoch-bound}.
\end{proof}

\begin{proof}[Proof of Proposition \ref{prop:abandon-patience-relation}]
According to Lemma~\ref{lemma:centralizedindicationfunction}, it suffices to show that as $n\to\infty$,
 \begin{equation*}
   \sup_{0\leq t\leq T}
   \Big|
    \tilde{G}^n(t)-\frac{1}{\sqrt{n}}\sum_{i=1}^{E^n(t)}\id{\gamma^n_i\leq \omega^n_i}
   \Big|\Rightarrow 0.
 \end{equation*}
As a customer arriving at the system before time $\zeta^n(t)$  must
have either entered service or abandoned the queue by time $t$, we have the
following relationship:
\begin{equation*}
  \sum_{i=1}^{E^n(\zeta^n(t)-)}\id{\gamma^n_i\leq \omega^n_i}\leq G^n(t)
  \leq \sum_{i=1}^{E^n(t)}\id{\gamma^n_i\leq \omega^n_i}.
\end{equation*}
Hence, it is enough to prove that as $n\to\infty$,
\begin{equation}\label{2012-march-26-9}
  \sup\limits_{0\leq t\leq T}
  \frac{1}{\sqrt{n}}\sum\limits_{i=E^n(\zeta^n(t)-)+1}^{E^n(t)}
  \id{\gamma^{n}_i \leq \omega^{n}_i }\Rightarrow 0.
\end{equation}
Note that
\begin{equation}\label{2012-march-26-8}
  \begin{split}
   \sup\limits_{0\leq t\leq T}
   \frac{1}{\sqrt{n}}\sum\limits_{i=E^n(\zeta^n(t)-)+1}^{E^n(t)}
   \id{\gamma^{n}_i \leq \omega^{n}_i }
   =& \sup\limits_{0\leq t\leq T}
   \frac{1}{\sqrt{n}}\sum\limits_{i=E^n(\zeta^n(t)-)+1}^{E^n(t)}
   (\id{\gamma^{n}_i \leq \omega^{n}_i }-F^{n}(\omega^{n}_i ))\\
   &+ \sup\limits_{0\leq t\leq T}
   \frac{1}{\sqrt{n}}\sum\limits_{i=E^n(\zeta^n(t)-)+1}^{E^n(t)}F^{n}(\omega^{n}_i).
  \end{split}
\end{equation}
By Lemma~\ref{lemma:centralizedindicationfunction}, the first term
on the right-hand side of \eqref{2012-march-26-8}
will converge to 0. For the second term,
\begin{equation}\label{eq:tech-ineq-arr-abd}
  \sup_{0\leq t\leq T}
  \frac{1}{\sqrt{n}}\sum_{i=E^n(\zeta^n(t)-)+1}^{E^n(t)}F^{n}(\omega^{n}_i)
  \leq \Big(\sup_{0\leq t\leq
  T}[\bar{E}^n(t)-\bar{E}^n(\zeta^n(t)-)]\Big)\cdot\Big(
  \sup_{0\leq i\leq E^n(T)}\sqrt{n}F^{n}(\omega^{n}_i)\Big),
\end{equation}
which weakly converges to 0 due to \eqref{eq:arrival-fluid},
Proposition~\ref{prop:waiting-time-stoch-bound} and Lemma
\ref{lem:timediff}. Thus, \eqref{2012-march-26-9} holds and the
proof is completed.
\end{proof}

\begin{proof}[Proof of Proposition \ref{prop:little-law}]
First note that $\omega^n(t) = \tilde \omega^n(t)/\sqrt n$. It
directly follows from  the stochastic boundedness of $\{\tilde
\omega^n, n\in \Z_+\}$
(Proposition~\ref{prop:waiting-time-stoch-bound}) that
\begin{eqnarray}
\sup\limits_{0\leq t\leq T} \omega^n(t)\Rightarrow 0.
\label{2014-7-8}
\end{eqnarray}

 By the
definition of $\omega^n(t)$, we have
\begin{equation}
  \label{eq:tech-basic-1}
  \begin{split}
    Q^n(t+\omega^n(t)) &\leq E^n(t+\omega^n(t))-E^n(t)\\
    &\leq Q^n((t+\omega^n(t))-)
      +\Big(E^n(t+\omega^n(t))-E^n(t+\omega^n(t)-\frac{1}{n})\Big)\\
      &\quad +\sum^{E^n(t+\omega^n(t))}_{i=E^n(t)}\id{\gamma^n_i \leq \omega^n_i}.
  \end{split}
\end{equation}
Note that
\begin{align}
  \label{2011-march-26-2}
  \frac{1}{\sqrt n}\left(E^n(t+\omega^n(t))-E^n(t)\right)
  &=\tilde E^n(t+\omega^n(t))-\tilde E^n(t)
  + \frac{\lambda^n}{n} \cdot \sqrt n \cdot  \omega^n(t), \\
  \label{2012-march-26-3}
  \frac{1}{\sqrt n} \sum_{i=E^n(t)}^{E^n(t+\omega^n(t))}
  \id{\gamma^n_i \leq \omega^n_i}
  &=\frac{1}{\sqrt n} \sum_{i=E^n(t)}^{E^n(t+\omega^n(t))}
  \left(\id{\gamma^n_i \leq \omega^n_i}-F^n(\omega^n_i) \right)\nonumber\\
  &\quad +\frac{1}{\sqrt n} \sum_{i=E^n(t)}^{E^n(t+\omega^n(t))} F^n(\omega^n_i).
\end{align}
By \eqref{eq:arrival-diffusion} and  (\ref{2014-7-8}),
\begin{eqnarray}\label{2012-april-08-3}
  \sup_{0\leq t\leq T} |\tilde E^n(t+\omega^n(t))-\tilde E^n(t)|\Rightarrow 0,
 \ \textrm{ as }n\to\infty.
\end{eqnarray}
By \eqref{eq:lambda-limit} and Proposition~\ref{prop:waiting-time-stoch-bound},
\begin{equation}
  \label{eq:2}
  |\frac{\lambda^n}{n} \cdot \sqrt n \cdot  \omega^n(t)-\mu\tilde \omega^n(t)|
  \Rightarrow 0,
   \  \textrm{ as }n\to\infty.
\end{equation}
It follows from \eqref{eq:lambda-limit} and \eqref{eq:arrival-diffusion} that
\begin{align}
  &\quad\sup_{0\leq t\leq T}
  \frac{1}{\sqrt n}\Big|E^n(t+\omega^n(t))-E^n(t+\omega^n(t)-\frac{1}{n})\Big|\nonumber\\
  & \ \ \ \ \ \ \leq \sup_{0\leq t\leq T} \Big|
    \tilde E^n(t+\omega^n(t))-\tilde E^n(t+\omega^n(t)-\frac{1}{n})
  \Big| +\frac{\lambda^n}{\sqrt {n^3}}\nonumber\\
  & \ \ \ \ \ \ \Rightarrow 0\quad\textrm{as }n\to\infty.\label{2012-april-08-5}
\end{align}
Note that the inequality \eqref{eq:tech-ineq-arr-abd} also holds
with $(\zeta^n(t)-, t)$ replaced by $(t,t+\omega^n(t))$, so by
(\ref{2014-7-8}) 
as
$n\to\infty$,
\begin{equation}
  \label{2012-april-08-2}
  \sup_{0\leq t\leq T} \frac{1}{\sqrt n} \sum_{i=E^n(t)}^{E^n(t+\omega^n(t))} F^n(\omega^n_i)
  \Rightarrow 0.
\end{equation}
Lemma \ref{lemma:centralizedindicationfunction}, \eqref{2012-march-26-3} and \eqref{2012-april-08-2} imply that as $n\to\infty$,
\begin{equation}
  \label{2012-april-08-7}
  \sup_{0\leq t\leq T}\frac{1}{\sqrt n}
  \sum_{i=E^n(t)}^{E^n(t+\omega^n(t))}\id{\gamma^n_i \leq \omega^n_i}
  \Rightarrow 0.
\end{equation}
By condition \eqref{eq:modular-cont}, as $n\to\infty$,
\begin{equation}
  \label{eq:tech-q-t+omega-}
   \sup_{0\leq t\leq T}\Big|
     \tilde Q^n(t+\omega^n(t))-\tilde Q^n((t+\omega^n(t))-)
   \Big| \Rightarrow 0.
\end{equation}
Applying the above convergence
\eqref{2012-april-08-3}--\eqref{eq:tech-q-t+omega-} to the
inequality \eqref{eq:tech-basic-1} yields that as $n\to\infty$,
\begin{equation*}
  \sup_{0\leq t\leq T}\Big|
    \tilde Q^n(t+\omega^n(t))-\mu \tilde \omega^n(t)
  \Big|\Rightarrow 0.
\end{equation*}
By condition \eqref{eq:modular-cont} and (\ref{2014-7-8}),
as $n\to\infty$,
\begin{equation*}
  \sup_{0\leq t\leq T}\Big |\tilde Q^n(t+\omega^n(t))-\tilde Q^n(t)\Big | \Rightarrow 0.
\end{equation*}
Thus, the result of this proposition follows.
\end{proof}

\end{document}